\documentclass[reqno]{amsart}
\usepackage[backend=bibtex,url=false]{biblatex}
\renewbibmacro{in:}{}
\usepackage{float}
\usepackage{amsmath}
\usepackage{amsthm}
\usepackage{graphicx}
\usepackage{latexsym}
\usepackage{amsfonts}
\usepackage{amssymb}
\usepackage{hyperref}
\theoremstyle{plain}
\newtheorem{theorem}{Theorem}
\newtheorem{corollary}[theorem]{Corollary}
\newtheorem{lemma}[theorem]{Lemma}
\newtheorem{proposition}[theorem]{Proposition}
\theoremstyle{definition}

\newtheorem{example}[theorem]{Example}

\newtheorem{remark}[theorem]{Remark}
\newtheorem*{remark*}{Remark}

\renewcommand{\Pr}{\mathbf P}

\newcommand{\ee}{\varepsilon}
\newcommand{\ff}{\varphi}
\newcommand{\oo}{\overline}
\newcommand{\uu}{\underline{g}}

\begin{document}

\title[First-passage times for walks with non-iid increments]
{First-passage times for random walks with non-identically distributed increments.}


\author[Denisov]{Denis Denisov}
\address{School of Mathematics, University of Manchester, UK}
\email{denis.denisov@manchester.ac.uk}

\author[Sakhanenko]{Alexander Sakhanenko}
\address{Novosibirsk State University, 
 630090 Novosibirsk, Russia}
\email{aisakh@mail.ru}

\author[Wachtel]{Vitali Wachtel}
\address{Institut f\"ur Mathematik, Universit\"at Augsburg, 86135 Augsburg, Germany}
\email{vitali.wachtel@mathematik.uni-augsburg.de}


\begin{abstract}
We consider random walks with independent but not necessarily identical
distributed increments. Assuming that the increments satisfy the well-known Lindeberg
condition, we investigate the asymptotic behaviour of first-passage times over moving
boundaries. Furthermore, we prove that a properly rescaled random walk conditioned to
stay above the boundary up to time $n$ converges, as $n\to\infty$, towards the
Brownian meander.
\end{abstract}



\maketitle

\section{Introduction and main results}

\subsection{Introduction}
Let $X_k$, $k\geq1$, be independent random variables and consider a random walk
$$
S_n:=X_1+X_2+\ldots+X_n,\ n\geq1.
$$
For a real-valued sequence  $\{g_n\}$  let
\begin{equation}
\label{def2}
T_g:=\min\{n\geq1:S_n\leq g_n\}
\end{equation}
be the first crossing of the moving boundary $g_n$ by $S_n$.
The main purpose of the present paper is to study the asymptotic behaviour of the distributions of
first-passage times over moving boundaries
\begin{equation*}                                           
  \Pr(T_g>n), \quad n\to\infty,
\end{equation*}
for random walks with non-identically distributed increments
in the domain of attraction of the Brownian motion.
An important particular case of this problem is the case of a constant
boundary $g_n\equiv-x$ for some $x$.
In this case $T_g\equiv\tau_x$, where
$$
\tau_x:=\min\{n\geq1:S_n\leq -x\}.
$$

If all $X_k$'s have identical distribution and  $S_n$ is oscillating then the problem
of finding asymptotics
$$
\Pr(\tau_x>n),\quad n\to \infty,
$$
has attracted considerable attention and is well understood.
In this case the following elegant  result
(see Doney \cite{Don95}) is available for asymptotically stable
random walks: if
\begin{equation*}                                        
\Pr(S_n>0)\to\rho\in(0,1)
\end{equation*}
then, for every fixed $x\geq0$,
\begin{equation}                           \label{iid}
\mathbf{P}(\tau_x>n)\sim V(x) n^{\rho-1}L(n),
\end{equation}
where $V(x)$ denotes the renewal function corresponding to the weak descending ladder height process.
(Here and in what follows all unspecified limits are taken with respect to $n\to\infty$.)

In particular, if $\mathbf{E}X_1=0$ and $\mathbf{E}X_1^2<\infty$ (we are still in the i.i.d. case) then the ladder heights have finite expectations and,
consequently, for every fixed $x\geq0$,
\begin{equation}
\label{finite_variance}
\mathbf{P}(\tau_x>n)\sim\sqrt{\frac{2}{\pi}}\frac{\mathbf{E}[-S_{\tau_x}]}{\sqrt{n}}.
\end{equation}

The use of the  Wiener-Hopf factorisation is a traditional approach to
derivation of \eqref{iid} and \eqref{finite_variance}.
In turn, the Wiener-Hopf factorisation essentially relies on the following important properties:
\begin{itemize}
\item[(a)] duality relation: if
$X_1,X_2,\ldots,X_n$ are independent and identically distributed then the distribution of random path
$\{S_k,k\leq n\}$ coincides with that of
$\{S_n - S_{n-k}; k\le n\}$ after duality transformation;
\item[(b)] simple geometry of semi-infinite intervals of the real line, which is well adapted to the
duality transformation.
\end{itemize}

Now what if the increments $X_k$  have different distributions, as we assume in this paper?
Clearly one loses the duality property and therefore there is no hope to generalise the factorisation
approach via the Wiener-Hopf identities to such random walks.
Moreover, when we consider  moving boundaries  the benefits of the simple geometry of
fixed semi-infinite intervals are no longer available.
Naturally  this leads to the following
question: how can one investigate first-passage times of random walks with non-identically
distributed increments? In the present paper we suggest to use the {\em universality approach}.

The suggested  approach is based on the  universality of  the Brownian motion that attracts
random walks with the finite variance.
To see the connection between boundary problems for
random walks and the Brownian motion
consider a similar problem for the Brownian motion and
define for each $x>0$ the stopping time
$$
\tau_x^{bm}:=\inf\{t>0:x+W(t)\leq0\}.
$$
Then, for every fixed $x>0$,
$$
\mathbf{P}(\tau_x^{bm}>t)\sim\sqrt{\frac{2}{\pi}}\frac{x}{\sqrt{t}}, \quad t\to\infty.
$$
Noting that the continuity of paths of the Brownian motion yields the equality
$x=\mathbf{E}[-W(\tau^{bm}_x)]$, we obtain
\begin{equation}
\label{bm}
\mathbf{P}(\tau_x^{bm}>t)\sim\sqrt{\frac{2}{\pi}}\frac{\mathbf{E}[-W(\tau^{bm}_x)]}{\sqrt{t}},
\quad t\to\infty.
\end{equation}
Comparing \eqref{finite_variance} and \eqref{bm}, we see that the asymptotic  behaviour of the tail of $\tau_x$
for any random walk with i.i.d. increments having zero mean and finite variance
coincides, up to a constant, with that of $\tau_x^{bm}$. Having this in mind one may assume that a
version of \eqref{finite_variance} should be valid for all random walks from the normal domain of
attraction of the Brownian motion.

We will now briefly indicate how we can use universality of the Brownian motion to establish
\eqref{finite_variance}. Consider   the easier case 
when random walk crosses the level
$-x_n=-uB_n$, where $u>0$ is a fixed
number and $B_n$ is the norming sequence in the
Functional Central Limit Theorem(FCLT).
Then, by the FCLT , we have the relation
\begin{align*}
\mathbf{P}(\tau_{x_n}>n)&=\mathbf{P}(x_n+\min_{k\le n}S_k>0)=\mathbf{P}(u+\min_{k\le n}S_k/B_n>0)
\\
&\to \mathbf{P}(u+\min_{t\le1}W(t)>0)=
\mathbf{P}(\tau_{x_n}^{bm}>B_n).
\end{align*}
Since one always has a certain rate of convergence in the functional CLT, the same relation remains
valid for $u=u_n$ decreasing to zero sufficiently slow.
Namely, if $u_n$ goes to zero slower than the rate of convergence,  then for
$x_n=u_nB_n$ we have
$$
\mathbf{P}(\tau_{x_n}>n)\sim\mathbf{P}(\tau_{x_n}^{bm}>B_n)
\sim\sqrt{\frac{2}{\pi}}\frac{x_n}{B_n}.
$$

It is not at all clear, how to use the FCLT in the case of a fixed~$x$.
In this case a direct application of the universality  results in  significant errors due to the
FCLT approximation.  However, this method becomes applicable when supplemented with
probabilistic understanding of the typical behaviour of a random walk staying
above $g_n$ for a long time.

The universality approach  to the analysis of the asymptotics for first passage times
is a far more general method than the Wiener-Hopf factorisation.
It has already been used in several instances, where the Wiener-Hopf method does not seem to be applicable
because of either the complex  geometry and/or problems with duality.
\begin{itemize}

\item Ordered random walks \cite{DW10}, \cite{KS10}, \cite{DW12}.
These papers studied the exit times of multidimensional random walks
from Weyl chambers.

\item Random walks in cones \cite{DW15}, where the exit times of multidimensional random walks
from general cones were studied.

\item Integrated random walks \cite{DW15b}, where a two-dimensional Markov chain was considered
to study exit times for integrated random walk.

\item Conditioned limit theorems for products of random matrices, see \cite{GPP16}.

\item Limit theorems for Markov walks conditioned to stay positive, see \cite{GPP16a} and \cite{GPP16b}.

\end{itemize}
Besides asymptotic results we can use
the universality approach to construct conditioned processes and prove functional limit theorems
for conditioned process.

There are 4 main steps in the universality approach used in the above papers:
\begin{enumerate}
  \item Show the repulsion from the boundary, which allows the random walks to reach quickly the high level
  of order $B_n^{1-\varepsilon}$
  \item Use the repulsion and recursive estimates to show the finiteness of
  expectation of the overshoot over the high level.
  \item Use strong coupling (KMT) to replace the trajectory of a random walk with the Brownian motion
  after the reaching of the high level. Apply asymptotics for the crossing time by the Brownian motion.
  \item Use the finiteness of the expectation of the overshoot for the additional control of the error
  in the approximation.
\end{enumerate}
The method is potentially applicable to the analysis of a large class of stochastic processes.
However, the main restriction of the method  was the necessity to use a strong coupling,
which is difficult to prove and is rarely available. For example, papers
 \cite{GPP16}, \cite{GPP16a} and \cite{GPP16b}  depend  on \cite{GPP14}, where an FCLT with a rate
of convergence (strong coupling) was proved.
The present paper deals with this deficiency and allows one to use directly the FCLT instead
of the strong coupling.
This is an important {\em  methodological novelty} of the present paper besides a number of
a new results.
  FCLT holds in a number of situation and we plan to develop the methodology  further
to study exit times (including higher dimensions) for  other stochastic processes.


\subsection{Statement of main results}
We shall always assume that
$$
\mathbf{E} X_k=0\ \text{ and }\ 0<\sigma_k^2:=\mathbf{E}X_k^2<\infty\ \text{ for all }k\geq1.
$$
Define $S_0=B^2_0=0$ and
$$
B^2_n:=\sum_{k=1}^n \sigma_k^2
,\quad n\geq1.
$$

About real numbers $\{g_n\}$ used in definition \eqref{def2} we assume that
\begin{equation}
\label{good}
 g_n=o(B_n)
\end{equation}
and
\begin{equation}
\label{must}
\mathbf{P}(T_g>n)>0\quad\text{for all }n\geq1.
\end{equation}
It is worth mentioning that assumption \eqref{must} is equivalent to the following condition
$$
\sum_{k=1}^n {\rm essup} X_k>g_n\quad\text{for all }n\geq1,
$$
where ${\rm essup} X_k:=\sup\{x:\mathbf{P}(X_k\geq x)>0\}$.
\hfill$\diamond$

To formulate our main results we introduce the classical random broken line process
\begin{gather}\label{T4.0}
s(t)=S_k+X_{k+1}\frac{(t-B^2_k)}{\sigma^2_{k+1}}\ \ \text{ for }\ \ t\in[B^2_k,B^2_{k+1}], \ \ k\geq1.
\end{gather}
We always consider
\begin{equation}\label{T4.0.scaled}
s_n(t):={s(tB_n^2)}/{B_n}
\end{equation}
as random process defined for $ t\in[0,1]$ with values in the space $C[0,1]$ of
continuous functions endowed with the supremum norm.
It is well known that the \emph{Lindeberg condition}
\begin{equation}
\label{lind.cond}
L_n^2(\varepsilon):=\frac{1}{B_n^2}\sum_{k=1}^n\mathbf{E}[X_k^2;|X_k|>\varepsilon B_n]\to 0\quad
\text{for every }\varepsilon>0
\end{equation}
is necessary and sufficient for the validity of the FCLT for
$s_n(\cdot)$.

\begin{theorem}                                                                                                 \label{T1}
Assume that conditions \eqref{good}, \eqref{must} and \eqref{lind.cond} hold. Then  the distribution of the process $s_n(\cdot)$,
conditioned on $\{T_g>n\}$, converges weakly on $C[0,1]$ towards the Brownian meander. In particular,
\begin{equation}
\label{joint}
\mathbf{P}\left(S_n>g_n+vB_n\big|T_g>n\right)\to e^{-v^2/2} \quad\text{ for\ all }\  v\ge0.
\end{equation}
\end{theorem}

Relation \eqref{joint} and the functional limit theorem generalise corresponding results of Greenwood
and Perkins~\cite{GP83,GP85}, where the case of i.i.d. increments satisfying
$\mathbf{E}[X_1^2\log(1+|X_1|)]<\infty$ and monotone decreasing boundaries has been considered.
In the case of i.i.d. increments and constant boundaries these limit theorems have been obtained by
Bolthausen~\cite{Bolt76}. We are not aware of any similar results for random walks with non-identically
distributed increments.

\begin{theorem}                                                                                                 \label{T2}
Under assumptions of Theorem \ref{T1},
\begin{equation}
\label{main}
\mathbf{P}(T_g>n)\sim\sqrt{\frac{2}{\pi}}\frac{U_g(B_n^2)}{B_n},
\end{equation}
where $U_g$ is a positive, slowly varying function with the values
\begin{equation}
\label{main+}
0<U_g(B_n^2)=\mathbf{E}[S_n-g_n;T_g> n]
\sim\mathbf{E}[-S_{T_g};T_g\le n].
\end{equation}
\end{theorem}

Asymptotic formula \eqref{main} generalises \eqref{finite_variance} to all random walks satisfying
the Lindeberg condition and to all boundaries satisfying
\eqref{good} and \eqref{must}. For homogeneous in time random
walks Novikov \cite{Novikov81,Novikov83} and Greenwood and Novikov \cite{GN87} have found conditions on
$g_n$ under which one has a version of \eqref{main} with a positive constant instead of $U_g$.

For  arbitrary  $t\in[B^2_k,B^2_{k+1}]$ we define function $U_g$ in the following natural way
\begin{equation}                                                      \label{trivial}
U_g(t):=U_g(B_k^2)+
\frac{(t-B^2_k)}{\sigma^2_{k+1}}\left(U_g(B_{k+1}^2)-U_g(B_k^2)\right).
\end{equation}
 Note also that \eqref{main} implies trivially that
\begin{equation}
\label{very_trivial}
\log\mathbf{P}(T_g>n)\sim -\log B_n.
\end{equation}
\begin{example}
One of the simplest cases of walks with non-identically distributed increments are weighted random walks.
Let $\{\xi_k\}$ be independent, identically distributed random variables with zero mean and unit
variance. And let $\{a_k\}$ be a sequence of positive numbers. We consider weighted increments $X_k=a_k\xi_k$. If
$$
\frac{
a_n^2}{\sum_{k=1}^n a_k^2}\to0
$$
then the Lindeberg condition is fulfilled and we may apply Theorem \ref{T1} to the walk with weights
$\{a_k\}$. In particular, if $a_n=n^{p+o(1)}$ for some $p>-1/2$ then $B_n^2=n^{2p+1+o(1)}$
and hence, by \eqref{very_trivial},
\begin{equation*}                   
\frac{\log\mathbf{P}(T_g>n)}{\log n}\to -p-\frac{1}{2}.
\end{equation*}
This improves Theorem 1.2 from Aurzada and Baumgarten \cite{AB11}, where the case of $g_n\equiv0$
has been considered under the assumptions $c_1 k^p\leq a_k\leq c_2 k^p$ for all $k$ and
$\mathbf{E} e^{\lambda|\xi_1|}<\infty$ for some $\lambda>0$.

Moreover, if we aditionally assume that $a_n=n^p\ell(n)$, where $\ell$ is a slowly varying function, then
$B_n\sim\frac{n^{p+1/2}\ell(n)}{\sqrt{2p+1}}$
and, consequently,
$$
\mathbf{P}(T_g>n)\sim\frac{L_g(n)}{n^{p+1/2}},
$$
where $L_g$ is slowly varying.

Using \eqref{very_trivial} one can obtain logarithmic asymptotics for
$\mathbf{P}(T_g>n)$ also for faster growing weight sequences. If, for example,
$a_n=\exp\{n^\alpha\ell(n)(1+o(1)\}$ with some $\alpha\in(0,1)$ then
$\log\mathbf{P}(T_g>n)\sim n^\alpha\ell(n)$.
\hfill$\diamond$
\end{example}

\begin{remark}\label{R.bm_g}
It will be clear from the proofs of Theorems~\ref{T1} and \ref{T2} that our approach applies also to the
Brownian motion. If $g$ is a continuous function with $g(0)<0$ and $|g(t)|=o(\sqrt{t})$ then
\begin{equation}
\label{bm_g}
\mathbf{P}(T^{bm}_g>t)\sim\frac{\ell_g(t)}{\sqrt{t}}\quad\text{as }t\to\infty,
\end{equation}
where
$$
T^{bm}_g:=\inf\{t\geq0:W(t)=g(t)\}.
$$
Relation \eqref{bm_g} improves results from Novikov \cite{Novikov81} and Uchiyama \cite{U80}.

Furthermore, the distribution of $\{W(ut)/\sqrt{t};u\in[0,1]\}$ conditioned on $\{T^{bm}_g>t\}$
converges, as $t\to\infty$, weakly on $C[0,1]$ towards the Brownian meander.
\hfill$\diamond$
\end{remark}
\subsection{Asymptotic behaviour of $U_g$}
Theorems \ref{T1} and \ref{T2} state that for any random walk belonging to the domain of attraction of the
Brownian motion and for any boundary sequence $g_n=o(B_n)$,
with necessary condition \eqref{must}, we have universal limiting
behaviour of conditional distributions. In \eqref{main} we also have the universal leading
term: $B_n^{-1}$. And dependence on the boundary $\{g_n\}$ and on the distribution of the
increments $\{X_k\}$ concentrates in the function $U_g$ only. In order to obtain exact asymptotics
for $\mathbf{P}(T_g>n)$ we have to determine the asymptotic behaviour of $U_g$.

Here we want to present conditions  (necessary and/or sufficient) under which
 the function $U_g(t)$
 have finite and/or positive limit as~$t\to\infty$.
Our simplest result is as follows.
\begin{theorem}                                                                                                          \label{T3}
Suppose that  all assumptions  of Theorem \ref{T1} are fulfilled and
\begin{equation}                                                                                                        \label{bndd}
\overline{g}:=\sup_{n}g_n<\infty.
\end{equation}
Then  the expectation $\mathbf{E}[-S_{T_g}]$
and the limit $\lim_{t\to\infty}U_g(t)$ are defined and
\begin{equation}                                                                                                        \label{lim+}
0<U_g(\infty):=\lim_{t\to\infty}U_g(t)=\mathbf{E}[-S_{T_g}]
=\mathbf{E}[\overline{g}-S_{T_g}]-\overline{g}\le\infty.
\end{equation}

In addition,  if for some integer $M$ sequence $\{g_n\}$ is non-increasing for all $n\ge M$ then the function $U_g(t)$ is non-decreasing for $t\ge B_M^2$.
\end{theorem}

Here we use the fact that mathematical expectation of any non-negative random variable is always defined but may be equal to infinity.

In the following two theorems we investigate the case when
\begin{equation}                                                                                                        \label{suplim}
{U}_g(\infty)=\lim_{t\to\infty}U_g(t)<\infty.
\end{equation}
It is worth mentioning that the study of $U_g$ simplifies significantly in the case when
boundary $g_n$ is non-increasing. In order to use this fact we introduce decreasing envelopes
of the sequence $\{g_n\}$:
\begin{equation}                                                                                                        \label{gMonot}
\min_{k\leq n}g_k=:\underline{g}_n\leq g_n\leq \overline{g}_n:=\sup_{k\geq n}g_k\leq\infty, \quad n\geq1.
\end{equation}

\begin{theorem}                                                                                                          \label{T4}
Suppose that conditions \eqref{bndd} and \eqref{suplim}    are fulfilled together with  all assumptions  of Theorem \ref{T1}.
 Then, with necessity,
\begin{equation}                                                                                                  \label{Lind+}
\sum_{n=1}^\infty\frac{1}{B_n}\mathbf{E}[-X_n;-X_n>\varepsilon B_n]<\infty
\quad\text{for each }\varepsilon>0
\end{equation}
and
\begin{equation}                                                                                                  \label{Sum-}
\sum_{n=2}^\infty\frac{\sigma_n^2}{B_n^3}(\overline{g}-\overline{g}_n)<\infty.
\end{equation}
\end{theorem}

Below, in Example \ref{anti-Lind}, we will show that condition  \eqref{Lind+} does not follow from the assertions of Theorems  \ref{T1} and \ref{T2}.
\begin{theorem}                                                                                                          \label{T5}
Suppose that all assumptions  of Theorem \ref{T1} are satisfied  and
\begin{equation}                                                                                                        \label{Sum+}
\sum_{n=2}^\infty\frac{\sigma_n^2}{B_n^3}(\underline{g}_1-\underline{g}_n)<\infty.
\end{equation}
Assume in addition that there exists a non-decreasing sequence $\{h_n>0\}$ of positive numbers such that
\begin{equation}                                                                                             \label{hLind}
\sum_{n=1}^\infty\frac{1}{B_n}\mathbf{E}[-X_n;-X_n>h_{n}+g_{n-1}-\uu_{n}]<\infty
\end{equation}
and
\begin{equation}                                                                                        \label{hSum}
\sum_{n=1}^\infty\frac{\sigma_n^2}{B_n^3}h_n<\infty.
\end{equation}
Then  the expectation $\mathbf{E}|S_{T_g}|<\infty$
and the limit $\lim_{t\to\infty}U_g(t)$ exist and
\begin{equation}                                                                                                        \label{lim}
0\le U_g(\infty):=\lim_{t\to\infty}U_g(t)=\mathbf{E}[-S_{T_g}]<\infty.
\end{equation}
\end{theorem}

Note that, for all $n\ge1$,
\begin{equation*}
\mathbf{E}[-X_n;-X_n>h_{n}+g_{n-1}-\uu_{n}]\le \mathbf{E}[-X_n;-X_n>h_{n}].
\end{equation*}
Remark, that if for some integer $M$ sequence $\{g_n\}$ is non-increasing for all $n>M$
then conditions \eqref{Sum-} and \eqref{Sum+} are equivalent. Note also that if
$g_n=O({B_n}/{\log^{1+\gamma}B_n})$, for some $\gamma>0$, then \eqref{Sum-} and
\eqref{Sum+} take place, and if $h_n=O({B_n}/{\log^{1+\gamma}B_n})$ then \eqref{hSum}
is fulfilled. Thus we have proved
\begin{corollary}                                                                                            \label{Cor-gamma}
Suppose that condition \eqref{Sum+} together with all assumptions of Theorem \ref{T1} hold and in addition
\begin{equation}                                                                                                            \label{Lind+gamma}
\sum_{k=1}^\infty\frac1{B_k}
\mathbf{E}\left[-X_{k+1};-X_{k+1}>\frac{B_k}{\log^{1+\gamma}B^2_k}\right]<\infty
\end{equation}
for some $\gamma>0$ and some $C>0$. Then $\mathbf{E}|S_{T_g}|<\infty$ and \eqref{lim} is true.
\end{corollary}

\begin{remark} In the case $g_n\equiv-x$ some estimates for the overshoot can be obtained from
Arak~\cite{Arak74}. First, combining Lemma 1.6 from that paper with our Theorem~\ref{T2}, one
can easily get
$$
\mathbf{E}[-S_{\tau_x}]\leq C\sum_{k=1}^\infty\frac{U_g(B_k)}{B_k^3}\mathbf{E}|X_k|^3.
$$
Then, recalling that $U_g$ is slowly varying, we  conclude that the condition
$$
\sum_{k=1}^\infty\frac{\mathbf{E}|X_k|^3}{B_k^{3-\gamma}}<\infty\quad\text{for some }\gamma>0
$$
is sufficient for the finiteness of $\mathbf{E}[-S_{\tau_x}]$. Second, according to Lemma~1.7
in \cite{Arak74},
$$
B_n\mathbf{P}(\tau_x>n)\leq C\left(x+\max_{k\leq n}\frac{\mathbf{E}|X_k|^3}{\mathbf{E}X_k^2}\right).
$$
Letting $n\to\infty$ and combining \eqref{main} with \eqref{lim+}, we obtain
$$
\mathbf{E}[-S_{\tau_x}]\leq C\left(x+\sup_{k\geq 1}\frac{\mathbf{E}|X_k|^3}{\mathbf{E}X_k^2}\right).
$$
All these estimates contain third absolute moments of the increments, since the main purpose of
\cite{Arak74} is to derive a Berry-Esseen type inequality for the maximum of partial sums.
\hfill$\diamond$
\end{remark}

Now we consider several particular cases in Theorems~\ref{T4} and \ref{T5}.

\begin{example} \label{anti-Lind}
Let $X_n$ be a symmetric random variable with four values:
$$
\mathbf{P}(X_n=\pm\sqrt{n})=\frac{p_n}{2},\quad \mathbf{P}(X_n=\pm a_n)=\frac{1-p_n}{2},
$$
where
$$
p_n:=\frac{1}{n\log(2+n)}\quad\text{and}\quad a_n:=\sqrt{\frac{1-np_n}{1-p_n}}.
$$
Clearly, $\mathbf{E}X_n=0$ and $\mathbf{E}X_n^2=1$. Therefore, $B_n=\sqrt{n}$ for this sequence
of random variables.

Let us first show that this sequence satisfies the Lindeberg condition.
Fix some $\varepsilon\in(0,1)$ and note that $a_n<1$ for each $n\geq1$. Then, for every
$n>\varepsilon^{-2}$,
\begin{align*}
L^2_n(\varepsilon)=\frac{1}{n}\sum_{k=1}^n\mathbf{E}[X_k^2;|X_k|>\varepsilon\sqrt{n}]
=\frac{1}{n}\sum_{k\in(\varepsilon^2n,n]}kp_k=O(\log^{-1}n).
\end{align*}

In order to see that \eqref{Lind+} does not hold here, we choose $\varepsilon=1/2$. Then
$$
\sum_{k=2}^\infty\frac{1}{B_k}\mathbf{E}[-X_k;-X_k>B_k/2]
=\sum_{k=2}^\infty\frac{1}{\sqrt{k}}\sqrt{k}p_k=\sum_{k=2}^\infty\frac{1}{k\log(2+k)}=\infty.
$$
Applying now Theorem \ref{T4}, we conclude that $\mathbf{E}[-S_{T_g}]=\infty$ and, consequently,
$$
\sqrt{n}\mathbf{P}(T_g>n)\to\infty
$$
by Theorem \ref{T2} for any boundary $g_n=o(\sqrt{n})$ with $\overline{g}<\infty$.
\hfill$\diamond$
\end{example}
This example shows that assumptions of Theorem \ref{T1} are not sufficient for condition \eqref{Lind+} to hold.

\begin{example}
Let $\{\xi_k\}$ be a sequence of independent, identically distributed random variables with
the probability density function
$$
f(x)=|x|^{-3}\mathbb{I}\{|x|\geq1\}.
$$
This sequence is still in the domain of attraction of the standard normal distribution, but not
in the normal domain of attraction. Due to the symmetry of the distribution of these variables,
the probability $\mathbf{P}(\tau_0>n)=\mathbf{P}(T_0>n)$ that the corresponding random walk stays positive up to time $n$ is
asymptotically equivalent to $c/\sqrt{n}$
(see, for example, \cite[Chapter XII.7, Theorem 1a]{F71}).

Let us consider different truncations of these
increments. For every $n\geq1$ define
$$
X_n:=\xi_n\mathbb{I}\{|\xi_n|\leq \sqrt{n}\log^p(n+2)\},\quad p\in\mathbb{R}.
$$
Clearly, $B_n^2\sim n\log n$ as $n\to\infty$. Furthermore, it is not hard to see that the
Lindeberg condition holds for every $p<-1/2$. Note also that $\sqrt{n\log n}$ is also the
norming sequence for the random walk with increments $\{\xi_k\}$. In other words, we have the same
type of convergence towards Brownian motion for all random walks considered in this example.

If we take $p<-1/2$ then $\mathbf{P}\left(-X_n>B_n/\log^{2+2\gamma}B_n\right)=0$ for all sufficiently
large values of $n$ with any  $\gamma\in(0,-p-1/2)$. Therefore,
\eqref{Lind+gamma} holds and, consequently, $\mathbf{P}(\tau_x>n)\sim c/\sqrt{n\log n}$.
This means that the truncation has changed the tail of~$T_g$.

But if we choose $p>1/2$, then $\mathbf{E}[-X_n;-X_n>B_n]\sim B_n^{-1}$. Recalling that
$B_n\sim\sqrt{n\log n}$, we conclude that the series in \eqref{Lind+} is infinite. This implies
that $\mathbf{P}(T_g>n)\gg 1/\sqrt{n\log n}$.
\hfill$\diamond$
\end{example}

Comparing \eqref{Lind+gamma} and \eqref{Lind+}, we see that the difference consists only in logarithmic
correction terms. In order to study the influence of these corrections, we consider again weighted
random walks.

\begin{corollary}                                                                 \label{weighted}
Let $\{X_k=a_k\xi_k\}$ where $\{a_k\}$ is a sequence of positive numbers
and $\{\xi_k\}$ are independent, identically distributed random variables with zero mean and unit
variance.
Suppose that for $\gamma=-1$ and for some $\gamma>0$ the following condition holds
\begin{equation}                                                                          \label{w-1}
\sum_{k=1}^\infty \frac{a_k}{B_k} \mathbb{I}\left\{
x>\frac{B_k}{a_k\log^{1+\gamma}B_k}\right\}\sim f_\gamma(x)\to\infty
\qquad\text{as }\quad x\to\infty
\end{equation}
with some functions
$f_\gamma$.
Then  \eqref{Lind+gamma} is equivalent to the  assumption
\begin{equation}                                                                          \label{w-2}
\mathbf{E}[(-\xi_1)f_\gamma(-\xi_1);\xi_1<0]<\infty
\end{equation}
whereas \eqref{Lind+} is equivalent to
 $\mathbf{E}[(-\xi_1)f_{-1}(-\xi_1);\xi_1<0]<\infty$.
\end{corollary}

Indeed, for positive weights $\{a_n\}$ condition \eqref{Lind+gamma} reduces to
$$
\sum_{k=1}^\infty\frac{a_k}{B_k}
\mathbf{E}\left[-\xi_1,-\xi_1>\frac{B_k}{a_k\log^{1+\gamma}B_k}\right]<\infty.
$$
Then, applying the Fubini theorem, we infer that  the last condition is
equivalent to \eqref{w-2}. Similar calculations with $\gamma=-1$ imply that
\eqref{Lind+} is equivalent to $\mathbf{E}[(-\xi_1)f_{-1}(-\xi_1);\xi_1<0]<\infty$.

\begin{example} First, consider the case when $a_k=k^p$ with some $p>0$.
It is easy to see that
$$
B_n^2=\sum_{k=1}^n k^{2p}\sim
n^{2p+1}/(2p+1)\sim na_n/(2p+1)
$$
and that we may take $f_\gamma(x)=c(p)x\log^{1+\gamma}x$.
From this relation we infer that \eqref{Lind+gamma} reduces to
$$
\mathbf{E}[\xi_1^2\log^{1+\gamma}(-\xi_1);\xi_1<0]<\infty
$$
whereas \eqref{Lind+} is equivalent to $\mathbf{E}[\xi_1^2;\xi_1<0]<\infty$.
Therefore, in the case of regularly varying weights we have to assume slightly more than the finiteness
of the second moment.
\hfill$\diamond$
\end{example}

\begin{example}
The situation becomes very different in the case of Weibullian weights.
Indeed, assume that $a_k=\exp\{k^\alpha\}$, where $0<\alpha<1.$
Then, using the L'Hospital rule, we get
\begin{align*}
B_n^2 = \sum_{k=1}^n e^{2k^\alpha}
\sim \int_0^n e^{2x^\alpha}dx\sim \frac1{2\alpha} n^{1-\alpha} e^{2n^\alpha}
= \frac1{2\alpha} n^{1-\alpha} a_n^2.
\end{align*}
Hence, the sum in \eqref{w-1} is asymptotically equivalent to
\begin{align*}
  f_\gamma(x)&\sim  \frac1{\sqrt{2\alpha}}\sum_{k=1}^\infty \frac{1}{k^{(1-\alpha)/2}}
  \mathbb{I}\left\{x>\frac{k^{(1-\alpha)/2}}{\log^{1+\gamma} ( k^{1-\alpha} e^{k^\alpha} /{\sqrt{2\alpha}} )}\right\}\\
  &\sim \frac1{\sqrt{2\alpha}}\sum_{k=1}^\infty \frac{1}{k^{(1-\alpha)/2}}
  \mathbb{I}\left\{x>k^{\beta(\alpha, \gamma)}\right\},
\quad \beta(\alpha, \gamma)=(1-3\alpha-2\alpha \gamma)/2.
\end{align*}
It is not difficult to see that $\beta(\alpha, \gamma)<0$ and $f_\gamma(x) =\infty$ when $\alpha\ge 1/3$ and $ \gamma>0$.
Hence, condition \eqref{w-2} never holds in this  case.

On the other hand, if $\beta(\alpha, \gamma)>0$ then
$$
 f_\gamma(x) \sim \frac1{\sqrt{2\alpha}}\int_{0}^{x^{1/\beta(\alpha, \gamma)}} \frac{1}{t^{(1-\alpha)/2}}dt
 = \frac1{\sqrt{2\alpha}}\frac2{1+\alpha} x^{\frac{1+\alpha}{2\beta(\alpha, \gamma)}}.
$$
Thus, for $\alpha<1/3$ and sufficiently small $\gamma>0$ condition \eqref{w-2} becomes
$$
\mathbf{E}[(-\xi_1)^{1+\frac{1+\alpha}{2\beta(\alpha, \gamma)}};\xi_1<0]
=\mathbf{E}[(-\xi_1)^{1+\frac{1+\alpha}{1-3\alpha-2\alpha \gamma}};\xi_1<0]
<\infty.
$$
For $\gamma=-1$ note that the necessary condition \eqref{Lind+} reduces to
$$\mathbf{E}[(-\xi_1)^{1+\frac{1+\alpha}{1-\alpha}};\xi_1<0]<\infty.$$
So, we see that  condition \eqref{w-2}  and equivalent condition \eqref{Lind+gamma} are much more restrictive
in the case of Weibullian weights.
\hfill$\diamond$
\end{example}
\section{Proof of Theorems~\ref{T1} and \ref{T2}}
\label{Sect:proof_of_main_result}
Throughout the  remaining part of the paper we will assume that conditions of Theorem \ref{T1} hold everywhere
except Lemma \ref{Lup1} and \ref{Lup2}.
\subsection{Estimates in a boundary problem} 
The main purpose of this subsection is to derive appropriate estimates for $\mathbf{P}(T_g>n)$ using ideas from the FCLT. Define
\begin{equation}\label{A1}
Z_k:=S_k-g_k \quad\text{and}\quad Z^*_k=Z_k\mathbb{I}\{T_g>k\},\ k\geq1.
\end{equation}
For every $h>0$ and each $m\geq 1$ consider the stopping times
\begin{equation}
\label{A2}
\nu(h):=\inf \{k\ge1:Z_k>h\}\quad\text{and}\quad \nu_m:=\min\{\nu(B_m),m\}.
\end{equation}

To state the main result of this paragraph we introduce the notation
\begin{gather}
\label{A3}
G_n:=\max_{k\le n}|g_k|
\quad\text{and}\quad
\rho_n:=3\pi_n+2\frac{G_n}{B_n},
\end{gather}
where $\pi_n$ denotes the classical Prokhorov distance  (see Lemma \ref{L.pi} below for details) between the distributions on $C[0,1]$ of the
Brownian motion and the process $s_n(t)$
 defined in \eqref{T4.0.scaled}.

\begin{proposition}\label{CLT-estim}
Let integers $m,n$ satisfy
\begin{equation}                                                                  \label{A4}
B_m\le \frac{3}{5}B_n,\qquad 1\le m<n.
\end{equation}
Then
\begin{align}
\label{A5}
\nonumber
&\alpha_{m,n}:=\left|B_n{\mathbf{P}}(T_g>n)-2\ff(0){\mathbf{E}}Z_{\nu_m}^*\right|\\
&\hspace{1cm}\le\rho_nB_n{\mathbf{P}}(T_g>\nu_m)+2{\mathbf{E}}Z_{\nu_m}^*\frac{B_m^2}{B_n^2}
+ {\mathbf{E}}\big[Z_{\nu_m}^*;Z_{\nu_m}^*>3B_m\big],
\end{align}
where $\varphi$ stands for the density of the standard normal distribution.
\end{proposition}
We prepare the proof of this Proposition by a series of Lemmata.
Later on in this subsection we suppose that integers $k,m,n$ and real $y$ satisfy the conditions
\begin{equation}\label{A6}
1\le k\le m<n, \qquad 0\le y<\infty.
\end{equation}
Let
\begin{equation}\label{A7}
Q_{k,n}(y):=
\mathbf{P}\Big(y+\min_{k \le j\leq n}(Z_j-Z_k)>0\Big).
\end{equation}
With $\nu=\nu(B_m)$ we have,
\begin{align*}
\mathbf{P}(T_g>n)&=\mathbf{P}\Big(\min_{ j\leq n}Z_j>0\Big)\\
&=\mathbf{P}\Big(\nu\le m, T_g>\nu, Z_\nu+\min_{\nu \le j\leq n}(Z_j-Z_\nu)>0\Big)\\
&\hspace{0.5cm}+\mathbf{P}\Big(\nu> m, T_g>m, Z_m+\min_{m \le j\leq n}(Z_j-Z_m)>0\Big).
\end{align*}
Hence, by the strong Markov property at time $\nu_m=\min\{\nu,m\}$,
\begin{gather}                                                                                                   \label{A9}
\mathbf{P}(T_g>n)
=\mathbf{E}\left[Q_{\nu_m,n}(Z_{\nu_m});T_g>\nu_m\right]
=\mathbf{E}Q_{\nu_m,n}(Z^*_{\nu_m})
\end{gather}
since events $\{T_g>\nu_m\}$ and $\{Z^*_{\nu_m}>0\}$ coincide  and $Q_{\nu_m,n}(0)=0$.

The rest of the subsection is devoted to estimation of the functions  $Q_{k,n}$. We are going to use
the following property which may be considered as one of the definitions of the Prokhorov distance $\pi_n$.
\begin{lemma}
\label{L.pi}
For each $n\ge1$ we can define a random walk $\{S_k,\ k\geq1\}$ and a
Brownian motion $W_n(t),\ t\in[0,\infty)$, on a common
probability space so that
\begin{align*}                                      
\nonumber
&\mathbf{P}\left(\max_{0\le t\leq B^2_n}|s(t)-W_n(t)|>\pi_n  B_n\right)\\
&\hspace{1cm}=\mathbf{P}\left(\max_{0\le t\leq 1}|s_n(t)-W_n(tB_n^2)/B_n)|>\pi_n \right)
\leq \pi_n.
\end{align*}
\end{lemma}
This result follows from Strassen's result \cite{Strassen65} applied together with the Skorohod
lemma~\cite{Skorohod77} to the Wiener process $W_n(tB_n^2)/B_n)$.

\begin{remark} As it was shown in Theorem 1 in \cite{Sakh06} for each $\alpha>2$ and every $\ee_n>0$ it is possible to
construct a Wiener process $W_n(t)$ such that
$$
\mathbf{P}\left(\max_{t\leq B^2_n}|s(t)-W_n(t)|>C\alpha\varepsilon_n B_n\right)
\leq L^{(\alpha)}_n(\varepsilon_n),
$$
where $C$ is an absolute constant and
\begin{equation*}
L^{(\alpha)}_n(\ee):=
\sum_{k=1}^n \mathbf{E}\min\left\{\frac{|X_k|^\alpha}{(\ee B_n)^\alpha},\frac{X_k^2}{(\ee B_n)^2}\right\}
\end{equation*}
may be called ``truncated Lindeberg fraction of order $\alpha$''.

The function $L^{(\alpha)}_n$ is very useful in estimating the rate of convergence in the functional central limit
theorem for the random walk $S_n$.
It is known (see, for example, Remark 2 in \cite{Sakh06}) that the Lindeberg Condition \eqref{lind.cond} is equivalent to
\begin{equation*} 
L^{(\alpha)}_n(\ee)\to 0
\quad\text{for every } \ \ee>0.
\end{equation*}
Moreover, there exists a sequence $\varepsilon_n\to0$ such that
\begin{equation*}                   
L^{(\alpha)}_n(\varepsilon_n)\le\ee_n\to 0\quad\text{as }n\to\infty.
\end{equation*}
As a result,
\begin{equation*} 
\pi_n\le C\alpha\ee_n\to0,
\end{equation*}
and this relation is equivalent to the Lindeberg condition.
\end{remark}

To state the next Lemma  we introduce further notation.
For every $1\leq k<n$ we define
\begin{equation}
\label{A10}
B_{k,n}^2:=B_n^2-B_k^2>0\quad\text{and}\quad
\ee_{k,n}:=\frac{\pi_nB_n+G_n}{B_{k,n}}.
\end{equation}
(Recall that $G_n=\max_{k\leq n}|g_k|$.)
It is well known that
\begin{gather}                                                                                                       \label{A11}
\forall\ y\ge0\quad
Q(y):=\mathbf{P}\Big(y+\min_{t\leq 1}W(t)>0\Big)=\Psi(y):=2\int_0^{y}\varphi(x)dx.
\end{gather}
It is easy to see from \eqref{A11} that
\begin{gather}                                                                                                       \label{A11+}
|\Psi(x+ z)-\Psi(x)|\le2\ff(0)|z|\quad\text{for all real}\quad x,z.
\end{gather}
\begin{lemma}                                                                                                          \label{L.est1}
For all $1\leq k<n$ and $y\ge0$,
\begin{equation}
\label{A12}
\left|Q_{k,n}(y)-\Psi\Big(\frac{y}{B_{k,n}}\Big)\right|
\le\pi_n+4\ff(0)\ee_{k,n}.
\end{equation}
\end{lemma}

\begin{proof}
For every $1\leq k<n$ consider
\begin{align}
\label{b3}
\nonumber
q_{k,n}(y):=\mathbf{P}\left(y+\min_{k\le j\leq n}(S_j-S_k)>0\right)
=\mathbf{P}\left(y+\min_{B^2_k\le t\leq B^2_n}(s(t)-s(B^2_k))>0\right),
\end{align}
where  $s(t)$ is the random broken line defined in \eqref{T4.0}.
It follows from \eqref{A1} that, for all $1\le k\le j\le n$,
$$
|(Z_j-Z_k)-(S_j-S_k)|=|g_k-g_j|\le2G_n.
$$
Hence, for $Q_{k,n}$ defined in \eqref{A7}, we have
\begin{equation}
\label{A15}
q_{k,n}(y_-)\le Q_{k,n}(y)\le q_{k,n}(y_+),\quad\text{where}\quad y_\pm:=y\pm2G_n.
\end{equation}

On the other hand, it is easy to see that
\begin{align*}
\left|\min_{B^2_k\le t\leq B^2_n}(s(t)-s(B^2_k))-\min_{B^2_k\le t\leq B^2_n}(W_n(t)-W_n(B^2_k))\right|
\le2\max_{t\leq B^2_n}|s(t)-W_n(t)|,
\end{align*}
with $W_n(t)$ is the Wiener process introduced in Lemma \ref{L.pi}.
Applying Lemma \ref{L.pi}, we obtain
\begin{align}\label{A16}
\nonumber
q_{k,n}(y_+)
&\le\pi_n+ \mathbf{P}\Big(y_++\min_{B^2_k\le t\leq B^2_n}(W_n(t)-W_n(B^2_k))>-2\pi_nB_n\Big)\\
\nonumber
&=\pi_n+\mathbf{P}\Big(\frac{y_++2\pi_nB_n}{B_{k,n}}+\min_{t\leq 1}W(t)>0\Big)\\
&=Q\Big(\frac{y}{B_{k,n}}+2\ee_{k,n}\Big)+\pi_n,
\end{align}
where we used the fact that $W(t)=(W_n(tB_{k,n}^2)-W_n(B^2_k))/B_{k,n}$ is also a standard Wiener process.
Using the same arguments, we obtain
\begin{align}\label{A17}
q_{k,n}(y_-)\geq Q\Big(\frac{y}{B_{k,n}}-2\ee_{k,n}\Big)-\pi_n .
\end{align}

It is easy to see from \eqref{A11} and \eqref{A11+} that, for $x,\ee\ge0$,
$$
Q(x+\ee)=\Psi(x+\ee)\le\Psi(x)+2\ff(0)\ee
$$
and
$$
Q(x-\ee)\ge\Psi(x-\ee)\ge\Psi(x)-2\ff(0)\ee.
$$
So, with $x=y/B_{k,n}$ and $\ee=2\ee_{k,n}$ we have
$$
\Big| Q\Big(\frac{y}{B_{k,n}}\pm2\ee_{k,n}\Big)-\Psi\Big(\frac{y}{B_{k,n}}\Big)\Big|
\le4\ff(0)\ee_{k,n}.
$$
Applying this inequality together with \eqref{A15}--\eqref{A17} we immediately obtain \eqref{A12}.
\end{proof}

\begin{lemma}                                                                                                          \label{L.est2}
Under conditions  \eqref{A4} and  \eqref{A6},
\begin{equation}
\label{A20}
\big|\Delta^*_{k,n}(y)\big|
\leq\delta^*_{k,n}(y):= \rho_nB_n\mathbb{I}\{y>0\} +2y\frac{B_m^2}{B_n^2}+y\mathbb{I}\{y>3B_m\},
\end{equation}
where
\begin{equation}
\label{A20+}
\Delta^*_{k,n}(y):=B_nQ_{k,n}(y)-2y\ff(0).
\end{equation}

\end{lemma}

\begin{proof}
First of all note that if $m$ satisfies \eqref{A4} then, for $1\leq k\leq m$,
\begin{equation}
\label{A21}
B_{k,n}\ge B_{m,n}\ge\frac{4}{5}B_n,\quad \ff(0)\le\frac{2}{5},
\quad \pi_n+4\ff(0)\ee_{k,n}\le\rho_n.
\end{equation}
In the last relation we have used \eqref{A10} and \eqref{A3}.
Set
\begin{equation}\label{def_delta_n}
\delta_{k,n}(y):=B_n\Psi\Big(\frac{y}{B_{k,n}}\Big)-2y\ff(0).
\end{equation}
Next we will bound $\delta_{k,n}(y)$ for $y\ge 0$ from above and below.
Since $\Psi(y)\le2y\ff(0)$ for all $y\ge0$, we have the following upper bound
\begin{equation}
\label{A23}
\delta_{k,n}(y)\le2y\ff(0)\left(\frac{B_n}{B_{k,n}}-1\right)
\le \frac{y(B_n^2-B_{k,n}^2)}{B_{k,n}(B_{k,n}+B_n)}
\le\left(\frac{5}{4}\right)^2\frac{yB_k^2}{2B_{n}^2}\le\frac{yB_m^2}{B_{n}^2}.
\end{equation}

We will need two different lower bounds. First, it follows immediately from
\eqref{def_delta_n} that
\begin{equation}
\label{A22}
\delta_{k,n}(y) \ge-2y\ff(0)\ge -y\quad \forall\ y\ge0.
\end{equation}
Second,  definition \eqref{A11} and the inequality $\ff(x)\ge\ff(0)(1-x^2/2)$  yield for $y\ge0$,
\begin{gather*}  
\Psi(y)=2\int_0^{y}\varphi(x)dx\ge2\int_0^{y}\ff(0)(1-x^2/2)dx=2\ff(0)(y-y^3/6).
\end{gather*}
Then we have
\begin{align}
\nonumber
\delta_{k,n}(y)&\ge B_n\Psi\Big(\frac{y}{B_{n}}\Big)-2y\ff(0)\ge -B_n\frac{2\ff(0)}{6}\Big(\frac{y}{B_{n}}\Big)^3
\\ \label{A25}
&\ge -3\ff(0)\frac{yB_m^2}{B_{n}^2}\ge -2\frac{yB_m^2}{B_{n}^2}
\qquad\text{for all}\qquad y\in[0,3B_m].
\end{align}

It follows from inequalities  \eqref{A23}-- \eqref{A25} that
\begin{equation} \label{A26}
\left|B_n\Psi\Big(\frac{y}{B_{k,n}}\Big)-2y\ff(0)\right|
\le2\frac{yB_m^2}{B_{n}^2}+y\mathbb{I}\{y>3B_m\} \qquad \forall\ y\ge0.
\end{equation}
 On the other hand we obtain from  \eqref{A12} and \eqref{A21} that
\begin{equation}\label{A27}
\left|Q_{k,n}(y)-\Psi\Big(\frac{y}{B_{k,n}}\Big)\right|\le\rho_n\mathbb{I}\{y>0\}
\end{equation}
since $Q_{k,n}(0)=0=\Psi(0)$.
Combining  \eqref{A26} and \eqref{A27}, we immediately find \eqref{A20}.
\end{proof}

\begin{proof} [Proof of Proposition \ref{CLT-estim}]
It follows from \eqref{A9} and \eqref{A20+} that
$$
\left|B_n{\mathbf{P}}(T_g>n)-2\ff(0){\mathbf{E}}Z_{\nu_m}^*\right|
=\left|\mathbf{E}\Delta_{\nu_m,n}^*(Z_{\nu_m}^*)\right|.
$$
Hence, by Lemma \ref{L.est2},
\begin{align*}
\nonumber
\left|\mathbf{E}\Delta_{\nu_m,n}^*(Z_{\nu_m}^*)\right|
&\le \mathbf{E}\left|\Delta_{\nu_m,n}^*(Z_{\nu_m}^*)\right|
\le \mathbf{E}\delta_{\nu_m,n}^*(Z_{\nu_m}^*)\\
&=\rho_nB_n{\mathbf{P}}(Z_{\nu_m}^*>0)+2{\mathbf{E}}Z_{\nu_m}^*\frac{B_m^2}{B_n^2}
+ {\mathbf{E}}\big[Z_{\nu_m}^*;Z_{\nu_m}^*>3B_m\big].
\end{align*}
It is easy to see that the obtained  estimate coincides with \eqref{A5}
once we recall that ${\mathbf{P}}(T_g>\nu_m)={\mathbf{P}}(Z^*_{\nu_m}>0)$.
Thus, the proof of the Proposition is completed.
\end{proof}

\subsection{Martingale properties of the sequence $Z^*_n$}
In this subsection we are going to prove the following  assertions.
\begin{lemma}                                                                       \label{Lmart}
For all $m\ge1$ we have
\begin{align}
\label{M1}
\mathbf{E}Z^*_m&=-\mathbf{E}[S_{T_g};T_g\le m]-g_m\mathbf{P}(T_g>m)
\end{align}
and
\begin{align}
\label{M2}
\mathbf{E}Z^*_{\nu_m}&=-\mathbf{E}[S_{T_g};T_g\le\nu_m]-\mathbf{E}[g_{\nu_m};T_g>\nu_m].
\end{align}
\end{lemma}
\begin{corollary}\label{rem1}
For all $n\ge m\ge1$ we have
\begin{align}                                                                \label{M3}
\mathbf{E}Z^*_{\nu_m}-\mathbf{E}Z^*_n \le2G_n \mathbf{P}(T_g>\nu_m),
\quad
\mathbf{E}Z^*_m-\mathbf{E}Z^*_n \le2G_n \mathbf{P}(T_g>m),
\end{align}
\begin{align}                                                                                  \label{M4}
|\mathbf{E}Z^*_{\nu_m}-\mathbf{E}Z^*_n| \le\alpha_{m,n}^*:=2G_n \mathbf{P}(T_g>\nu_m)
+\mathbf{E}[-Z_{T_g};\nu_m<T_g\le n]
\end{align}
and
\begin{align}                                                                                  \label{M5}
\max_{m\le k\le n}|\mathbf{E}Z^*_k-\mathbf{E}Z^*_n| \le2G_n \mathbf{P}(T_g>m)
+\mathbf{E}[-Z_{T_g};m<T_g\le n]\le\alpha_{m,n}^*.
\end{align}
\end{corollary}
\begin{remark}\label{rem2}
If $\{g_n\}$ is non-increasing for all $n\ge M\ge1$ then the sequence $\{Z^*_n\}$ is a submartingale (for $n\ge M$) and, hence, sequence
$\{Z^*_n\}$ is  non-decreasing for $n\ge M$ whereas function $U_g(t)$ is non-decreasing when $t\ge B^2_M$.

Indeed,
setting $\mathcal{F}_n:=\sigma(X_1,X_2,\ldots, X_n)$, we have
\begin{align*}
\mathbf{E}[Z^*_{n+1}|\mathcal{F}_n]
&=\mathbf{E}[(S_{n+1}-g_{n+1}){\mathbb I}\{T_g>n+1\}|\mathcal{F}_n]\\
&=\mathbf{E}[(S_{n+1}-g_{n+1})({\mathbb I}\{T_g>n\}-{\mathbb I}\{T_g=n+1\})|\mathcal{F}_n]\\
&=\mathbf{E}[(S_{n+1}-g_{n+1})|\mathcal{F}_n]{\mathbb I}\{T_g>n\}
-\mathbf{E}[(S_{n+1}-g_{n+1}){\mathbb I}\{T_g=n+1\}|\mathcal{F}_n]\\
&=(S_{n}-g_{n+1}){\mathbb I}\{T_g>n\}
-\mathbf{E}[(S_{n+1}-g_{n+1}){\mathbb I}\{T_g=n+1\}|\mathcal{F}_n]\\
&=Z^*_n+(g_n-g_{n+1}){\mathbb I}\{T_g>n\}
+\mathbf{E}[(g_{n+1}-S_{n+1}){\mathbb I}\{T_g=n+1\}|\mathcal{F}_n].
\end{align*}
Since $g_{n+1}\ge S_{n+1}$  on the event $\{T_g=n+1\}$ and $g_n\geq g_{n+1}$ for all $n\ge M$,
we obtain the submartingale property.
\hfill$\diamond$
\end{remark}

\begin{proof}[Proof of Lemma \ref{Lmart}]
For any bounded stopping time $\nu\ge1$,
by the optional stopping theorem,
$$
0=\mathbf{E}S_{T_g\wedge\nu}
=\mathbf{E}[S_{T_g};T_g\leq\nu]+\mathbf{E}[S_{\nu};T_g>\nu].
$$
Therefore,
$$
\mathbf{E}[S_{\nu};T_g>\nu]=-\mathbf{E}[S_{T_g};T_g\leq\nu].
$$
From this equality and the definition of $Z^*_n$ we get
$$
\mathbf{E}[Z^*_{\nu}]=\mathbf{E}[(S_{\nu}-g_{\nu});T_g>\nu]=
-\mathbf{E}[S_{T_g};T_g\leq\nu]-\mathbf{E}[g_{\nu};T_g>\nu].
$$
Taking $\nu=\nu_m$ and $\nu=m$ we obtain respectively  \eqref{M1} and \eqref{M2}.
\end{proof}

\begin{proof}[ Proof of Corollary \ref{rem1}]
From \eqref{M1} (with $m:=n$) and \eqref{M2} we have
\begin{align*}
\mathbf{E}Z^*_{\nu_m}-\mathbf{E}Z^*_n&=\mathbf{E}[S_{T_g};\nu_m<T_g\le n]-\mathbf{E}[g_{\nu_m};T_g>\nu_m]+g_n\mathbf{P}(T_g>n)
\\
&=\mathbf{E}[Z_{T_g};\nu_m<T_g\le n]
+\mathbf{E}[g_{T_g}-g_{\nu_m};\nu_m<T_g\le n]\\
&\hspace{1cm}+\mathbf{E}[g_n-g_{\nu_m};T_g>n].
\end{align*}
This equality implies \eqref{M4} and the first estimate in \eqref{M3}  since $Z_{T_g}\le 0$
and $|g_k|\le G_n$ for all $k\le n$.

Similarly, using \eqref{M1} again with $m:=k$ and $m:=n$ we obtain
\begin{align*}
\mathbf{E}Z^*_k-\mathbf{E}Z^*_n&=\mathbf{E}[S_{T_g};k<T_g\le n]-g_k\mathbf{P}(T_g>k)+g_n\mathbf{P}(T_g>n)
\\
&=\mathbf{E}[Z_{T_g};k<T_g\le n]
+\mathbf{E}[g_{T_g}-g_{k};k<T_g\le n]\\
&\hspace{1cm}+(g_n-g_k)\mathbf{P}(T_g>n).
\end{align*}
This equality with $k=m$ implies the second estimate in \eqref{M3}.
In addition, for $n\ge k\ge1$,
\begin{gather*}
|\mathbf{E}Z^*_k-\mathbf{E}Z^*_n| \le2G_n \mathbf{P}(T_g>k)
+\mathbf{E}[-Z_{T_g};k<T_g\le n].
\end{gather*}
Noting that the right hand side in the last inequality is a non-increasing function of $k$ we obtain \eqref{M5}.
\end{proof}

\subsection{Upper bounds} 
It follows from \eqref{good} and the Lindeberg condition \eqref{lind.cond}  that
\begin{equation}                                                                                     \label{Up0}
\lambda_n:=\min\{\ee>0:L_n(\ee)\le\ee\}\to0
, \quad
\overline{\sigma}_n^2:=\max_{k\le n}\frac{\sigma_k^2}{B_n^2}\le2\lambda_n^2\to0
\end{equation}
and $\rho_n=3\pi_n+2G_n/B_n\to0$.
In particular,  these  relations imply
\begin{equation}                                                                                     \label{Up1}
N_1:=\max\left\{n:3\pi_n+2\frac{G_n}{B_n} +2\lambda_n^2>1/8 \right\}<\infty.
\end{equation}
Since $B_{n}^2=B_{n-1}^2+\sigma_n^2\le B_{n-1}^2+B_n^2\oo\sigma_n^2\le B_{n-1}^2+B_n^2/8$, we have
\begin{equation}                                                                                     \label{Up1+}
\sup_{n>N_1}\frac{B_n^2}{B_{n-1}^2}\le\frac{8}{7},
\qquad\sup_{n>N_1}\frac{G_n}{B_n}
\le\frac{1}{16}.
\end{equation}

In what follows symbols $N_1,N_2,\dots$ and $C_1,C_2,\dots$ denote finite positive
constants which may depend on the sequence of numbers $g=\{g_n\}$ and on the fixed
joint distribution of random variables $\{X_n\}$.

The main purpose of this
subsection is to prove the following estimates.
\begin{proposition}
\label{P.upper}
There exists an integer $N_2\ge N_1$ such that, for all $n>N_2$,
\begin{align}                                                                         \label{Up2}
B_{n}\mathbf{P}(T_g>n)< 3\mathbf{E}Z^*_n
\end{align}
and
\begin{align}                                                                        \label{Up3}
\left|\mathbf{E}Z^*_n+\mathbf{E}[S_{T_g};T_g\le n]\right|
\le 3G_n\frac{\mathbf{E}Z^*_n}{B_n}<\frac{\mathbf{E}Z^*_n}{4}.
\end{align}
In addition, for all $m,n$ such that
\begin{gather}                                                               \label{Up-cond}
n\ge m> N_2 \qquad\text{and}\qquad B_m\ge8 G_n
\end{gather}
we have
\begin{align}                                                                \label{Up5}
\mathbf{E}Z^*_m\le4\mathbf{E}Z^*_n,
\quad
\mathbf{E}Z^*_{\nu_m}\le6\mathbf{E}Z^*_n,
\quad
\mathbf{P}(T_g>\nu_m)\le20\frac{\mathbf{E}Z_{n}^*}{B_m},
\end{align}
and
\begin{align}                                                                         \label{Up6}
 \alpha_{m,n}^*
\le40(2G_n+\lambda_n B_n)\frac{\mathbf{E}Z^*_n}{B_m}.
\end{align}

\end{proposition}

We first prove two auxiliary results. 
The following one is an easy generalisation of Lemma 7 from Greenwood and Perkins \cite{GP83}.
\begin{lemma}
\label{Lup1}
If $X_1, X_2,\ldots,X_n$ are independent and $\mathbf{P}(T_g>n)>0$ for some $n\geq1$ then
$$
\mathbf{P}(S_n> x|T_g>g_n)\geq \mathbf{P}(S_n> x)\quad \forall\ x\in\mathbb{R}.
$$
\end{lemma}
\begin{proof}
The statement of the lemma is obvious for $x\le g_n$. Therefore, we shall always assume that $x>g_n$.
We are going to use induction. If $n=1$ then, for every $x>g_1$,
$$
\mathbf{P}(S_1> x|T_g>1)=\frac{\mathbf{P}(S_1>x,T_g>1)}{\mathbf{P}(T_g>1)}
=\frac{\mathbf{P}(S_1> x)}{\mathbf{P}(T_g>1)}\geq \mathbf{P}(S_1\geq x).
$$

Assume now that the inequality holds for $n$. For every $x> g_{n+1}$ we have
\begin{align*}
&\mathbf{P}(S_{n+1}> x|T_g>n+1)\\
&\hspace{0.5cm}=\frac{\mathbf{P}(T_g>n)}{\mathbf{P}(T_g>n+1)}
\int_{\mathbb{R}}\mathbf{P}(y+S_{n}> x, y+S_n>g_{n+1} \mid T_g>n)\mathbf{P}(X_{n+1}\in dy)\\
&\hspace{0.5cm}\geq
\int_{\mathbb{R}}\mathbf{P}(y+S_{n}> x \mid T_g>n)\mathbf{P}(X_{n+1}\in dy)\\
&\hspace{0.5cm}\geq
\int_{\mathbb{R}}\mathbf{P}(y+S_{n}> x )\mathbf{P}(X_{n+1}\in dy)
=\mathbf{P}(S_{n+1}>x).
\end{align*}


Thus, the proof is finished.
\end{proof}
\begin{lemma}
\label{Lup2}
If $X_1,X_2,\ldots,X_n$ are independent for some $n\ge 1$ then
\begin{equation}                                                                                             \label{Up7}
\mathbf{E}Z_{n}^+\mathbf{P}(T_g>n) \le \mathbf{E}Z_{n}^*.
\end{equation}
\end{lemma}
\begin{proof} If $\mathbf{P}(T_g>n) =0$ then inequality \eqref{Up7} is obvious.
If $\mathbf{P}(T_g>n) >0$ then by Lemma \ref{Lup1}
\begin{align*}
\mathbf{E}[Z^*_{n}\mid T_g>n] &= \mathbf{E}[S_{n}-g_n\mid T_g>n]
=\int_0^\infty \mathbf P(S_n>g_n+x \mid T_g>n ) dx \\
&\ge \int_0^\infty \mathbf P(S_n>g_n+x ) dx   = \mathbf{E}(S_{n}-g_n)^+
= \mathbf{E}Z_n^+.
\end{align*}
Therefore,
$
\mathbf{E}Z^*_{n}\ge \mathbf{P}(T_g>n)\mathbf{E}Z_{n}^+.
$
\end{proof}
Note that Lemmas \ref{Lup1} and  \ref{Lup2} are the only Lemmas in the
Section~2 in which we do not impose all assumptions of Theorem~\ref{T1}.

\begin{proof} [Proof of Proposition \ref{P.upper}]
By the Central Limit Theorem $Z_n/B_n$ converges in distribution to $ W(1)$.
Hence, applying Fatou's Lemma,  we have
$$
\liminf_{n\to\infty}\mathbf{E}Z_{n}^+/B_n\ge\mathbf{E}W(1)^+=\int_0^\infty x\ff(x)dx=\ff(0)>1/3.
$$
From this estimate and Lemma~\ref{Lup2} we conclude that \eqref{Up2} is valid with
$$
N_2:=\max\left\{n\ge N_1:\mathbf{E}Z_{n}^+\le B_n/3\right\}<\infty.
$$
Next, the first inequality in \eqref{Up3} follows  from \eqref{Up2} and \eqref{M1}.
The second one in \eqref{Up3} is a corollary of the second bound in  \eqref{Up1+}.

Now, by the Markov inequality,
\begin{equation}\label{d1}
\mathbf{P}(Z_{\nu_n}^*>B_n)\le \frac{\mathbf{E}Z_{\nu_n}^*}{B_n}.
\end{equation}
On the other hand,
\begin{equation}\label{d2}
\mathbf{P}(Z_{\nu_m}^*\in(0,B_m])=\mathbf{P}(\nu_m=m,T_g>m)\le
\mathbf{P}(T_g>m).
\end{equation}
As $\nu_m\le m$ (see \eqref{A2}), we obtain,
by combining \eqref{d1} and \eqref{d2},
\begin{gather}                                                                                            \label{Up9}
\mathbf{P}(T_g>m)\le\mathbf{P}(T_g>\nu_m)=\mathbf{P}(Z_{\nu_m}^*>0)\le
\frac{\mathbf{E}Z_{\nu_m}^*}{B_m}+\mathbf{P}(T_g>m).
\end{gather}
Using again \eqref{Up1+} we have  from  \eqref{M3} with $m=n$ that
\begin{gather*}
\mathbf{E}Z_{\nu_m}^*\le\mathbf{E}Z_{m}^*+\frac{B_m}{8}\mathbf{P}(T_g>\nu_m) .
\end{gather*}
This fact and \eqref{Up9} yield
\begin{gather*}
\mathbf{P}(T_g>\nu_m)\le\frac{\mathbf{E}Z_{m}^*}{B_m}
+\frac{1}{8}\mathbf{P}(T_g>\nu_m)+\mathbf{P}(T_g>m) .
\end{gather*}
Hence,
\begin{gather}                                                                                           \label{Up10}
\mathbf{P}(T_g>\nu_m)\le\frac{8}{7}\frac{\mathbf{E}Z_{m}^*}{B_m}
+\frac{8}{7}\mathbf{P}(T_g>m)<5\frac{\mathbf{E}Z_{m}^*}{B_m},
\end{gather}
where the last inequality follows from  \eqref{Up2}.

From \eqref{M3}, \eqref{Up2} and \eqref{Up-cond} we obtain
\begin{gather*}
\mathbf{E}Z^*_m-\mathbf{E}Z^*_n \le2G_n \mathbf{P}(T_g>m)\ \le 6G_n\frac{\mathbf{E}Z^*_m}{B_m}
\le \frac{3}{4}\mathbf{E}Z^*_m.
\end{gather*}
This proves the first inequality in \eqref{Up5}. Similarly,
\begin{align*} 
\mathbf{E}Z^*_{\nu_m}-\mathbf{E}Z^*_n &\le2G_n \mathbf{P}(T_g>\nu_m)
\le 10G_n\frac{\mathbf{E}Z^*_m}{B_m}
\\ \nonumber
 &\le \frac{10}{8}\mathbf{E}Z^*_m \le \frac{40}{8}\mathbf{E}Z^*_n=5\mathbf{E}Z^*_n,
\end{align*} %
which implies the second estimate in \eqref{Up5}.

At last, substituting the first estimate from \eqref{Up5} into \eqref{Up10},
we obtain the third inequality in \eqref{Up5}. So, all estimates in \eqref{Up5}
are proved. Finally, the last inequality  \eqref{Up6} follows from  \eqref{M4},
the third inequality in \eqref{Up5} and from the following Lemma.
\end{proof}

\begin{lemma}                                                             \label{Lup3}
For all  $n\ge m>N_1$,
\begin{gather}                                                                                                     \label{Up11}
\beta_{m,n}^*:=\mathbf{E}[-Z_{T_g};n\ge T_g>\nu_m]
\le 40(G_n+\lambda_n B_n)\frac{\mathbf{E}Z^*_n}{B_m}.
\end{gather}
\end{lemma}
\begin{proof} Note that $-Z_{T_g}=-Z_{T_g-1}-g_{T_g-1}+g_{T_g}-X_{T_g}< 2G_n- X_{T_g}$ because $-Z_{T_g-1}<0$.
Hence, for any $\ee>0$,
\begin{align}                                                                                                     \label{Up12}
\beta_{m,n}^*&\le2G_n\mathbf{P}(T_g>\nu_m)+\mathbf{E}[-X_{T_g};n\ge T_g>\nu_m]
\\    \nonumber          %
&\le(2G_n+\ee B_n)\mathbf{P}(T_g>\nu_m)+\mathbf{E}[-X_{T_g};-X_{T_g}>\ee B_n , n\ge T_g>\nu_m].
\end{align}
By the definition of $\nu_m$ (see \eqref{A2}), for $2\le j\le n$ we have
\begin{align*}   \nonumber
\beta_{j,m,n}&:=\mathbf{E}[-X_{T_g};T_g=j>\nu_m, -X_{T_g}> \ee B_n]
\\\nonumber
&\le\mathbf{E}[-X_j;T_g>j-1\ge\nu_m, -X_j>\ee B_n]\\\nonumber
&=\mathbf{E}[-X_j; -X_j>\ee B_n]\mathbf{P}(T_g>j-1\ge\nu_m)
\\
&\le\mathbf{E}[-X_j; -X_j>\ee B_n]\mathbf{P}(T_g>\nu_m)\\
&\le\mathbf{E}[X_j^2; |X_j|>\ee B_n]\frac{\mathbf{P}(T_g>\nu_m)}{\ee B_n}.
\end{align*}
It follows now from \eqref{Up12} that
\begin{align*}                                                           
\beta_{m,n}^*&\le(2G_n+\ee B_n)\mathbf{P}(T_g>\nu_m)+\sum_{j=2}^n\beta_{j,m,n}
\\
&\le(2G_n+\ee B_n)\mathbf{P}(T_g>\nu_m)+
\sum_{j=2}^n\mathbf{E}[X_j^2; |X_j|>\ee B_n]\frac{\mathbf{P}(T_g>\nu_m)}{\ee B_n}
\\    \nonumber          
&\le\left(2G_n+\ee B_n+\frac{L^2_n(\ee)B_n}{\ee} \right)\mathbf{P}(T_g>\nu_m).
\end{align*}
Minimizing with respect to $\ee>0$ (see \eqref{Up0}) we obtain
\begin{gather*}        
\beta_{m,n}^*\le(2G_n+2\lambda_n B_n)\mathbf{P}(T_g>\nu_m).
\end{gather*}
combining this with the last inequality in \eqref{Up5}, we obtain \eqref{Up11}.
\end{proof}
\subsection{Rate of convergence in Theorem \ref{T2}} 

We are going to prove Theorem~\ref{T2} and obtain the following rate of convergence in \eqref{main}.
\begin{theorem}                                                                                  \label{T7}
Under assumptions of Theorem \ref {T1} asymptotics \eqref{main} hold with function $U_g$ defined in \eqref{trivial} which is  slowly varying. Moreover, for all $n\ge1$
\begin{align}                                                                                                                       \label{oo2}
\alpha^*_n:=\left|B_n\frac{\mathbf{P}(T_g>n)}{\mathbf{E}Z^*_n}-2\ff(0)\right|
\le C_1\left(\rho_n^{2/3}+\lambda_n^{1/2}\right)\to0
\end{align}
for some $C_1<\infty$.
\end{theorem}

We split the proof 
into several steps.
Define
\begin{align}                                                                                                    \label{oo1}
m(n)&:=\min\left\{k\ge 1:B_k^2\ge (\rho_n^{2/3}+\lambda_n^{1/2})B_n^2 \right\}
\end{align}
and
\begin{align}                                                                                                   \label{oo1+}
N_3&:=\max\left\{n\ge N_2:\rho_n^{2/3}+\lambda_n^{1/2}+2\lambda_n^2>\left(3/5\right)^2\right\}<\infty.
\end{align}

\begin{lemma}                                                              \label{over3}
If $n>N_3$ then the number $m=m(n)$ defined in \eqref{oo1} satisfies
conditions \eqref{A4} and \eqref{Up-cond}. In addition, for all $n>N_3$,
\begin{gather}                                                                                                     \label{oo5}
\alpha_{n}^*\le72\left(\rho_n^{2/3}+\lambda_n^{1/2}\right)
+\frac{\beta_{m(n)}}{\mathbf{E}Z^*_n},
\end{gather}
where $\beta_m:=\mathbf{E}[Z^*_{\nu_m};Z^*_{\nu_m}> 3B_m]$.
\end{lemma}

\begin{proof}
First, consider integers $m,n$ which satisfy conditions \eqref{A4} and \eqref{Up-cond}.
Comparing definitions \eqref{A5} and \eqref{oo2} we obtain
\begin{gather*}                                                                       
\alpha_{n}^*\mathbf{E}Z^*_n\le \alpha_{m,n}
+2\ff(0)\left|\mathbf{E}Z^*_{\nu_m}-\mathbf{E}Z^*_n\right|
\le \beta_{m}+\delta_{m,n},
\end{gather*}
where, using \eqref{A5} and \eqref{M4} , we have
\begin{gather*}                                                                                                     
\delta_{m,n}=2{\mathbf{E}}Z_{\nu_m}^*\frac{B_m^2}{B_n^2}+\rho_nB_n{\mathbf{P}}(T_g>\nu_m)
+ \alpha_{m,n}^*.
\end{gather*}
Now from estimates \eqref{Up5} and \eqref{Up6} we obtain
\begin{gather}                                                                                                     \label{oo8}
\delta_{m,n}\le12\mathbf{E}Z_n^*\frac{B_m^2}{B_n^2}
+(20\rho_n+40\rho_n+40\lambda_n)\mathbf{E}Z^*_n\frac{B_n}{B_m}
\end{gather}
since $2G_n/B_n<\rho_n$.

Second, consider integer $m=m(n)$ from \eqref{oo1} with $n>N_3$.
We have from \eqref{Up0} and \eqref{oo1+} that
\begin{gather}                                                                                                    \label{oo11}
B_{m(n)}^2=B_{m(n)-1}^2+\sigma_{m(n)}^2
\le(\rho_n^{2/3}+\lambda_n^{1/2})B_n^2+2\lambda_n^2B_n^2
\le \left(3/5\right)^2B_n^2.
\end{gather}
 So, condition \eqref{A4} is fulfilled in this case.
Furtehrmore, it follows from \eqref{Up1} that $2G_n/B_n<\rho_n\le1/8$ for $n>N_3\ge N_1$.
Hence, by \eqref{oo1},
$$
B_{m(n)}\geq \sqrt[3]{\rho_n}B_n=\rho_nB_n/\rho_n^{2/3}\ge4\rho_nB_n
\ge4(2G_b/B_n)B_n=8G_n.
$$
So, $m(n)$ satisfies also the condition \eqref{Up-cond} and we may apply
Proposition \ref{P.upper}. Since $B_{m(n)}\ge\sqrt[3]{\rho_n}B_n$
and $B_{m(n)}\ge\sqrt[4]{\lambda_n}B_n$, we have from \eqref{Up6} and
\eqref{oo8} the bound
\begin{gather*}                                                                                                     
\delta_{m,n}\le12\mathbf{E}Z_n^*\frac{B_m^2}{B_n^2}
+(60\rho_n^{2/3}+40\lambda_n^{3/4})\mathbf{E}Z_n^*.
\end{gather*}
And, thus, by \eqref{oo11},
\begin{gather*}                                                                                                     
\delta_{m,n}\le(72\rho_n^{2/3}
+12\lambda_n^{1/2}+24\lambda_n^2+40\lambda_n^{3/4})\mathbf{E}Z_n^*.
\end{gather*}
So, \eqref{oo5} follows immediately because $\lambda_n\le1/4$ by \eqref{Up1}.
\end{proof}

\begin{lemma}                                                              \label{over4}
Function $U_g$ is slowly varying. In addition, there exists a constant $C_2<\infty$ such that
\begin{gather}                                                                 \label{oo15}
\mathbf{P}(T_g>j-1)\le  C_2\mathbf{E}Z^*_n\frac{B_n^{1/3}}{B_j^{4/3}}
\quad\text{for all }j\in[1,n].
\end{gather}

\end{lemma}
\begin{proof}
First note that, by \eqref{Up6} and \eqref{oo1},
\begin{align*}
\frac{\alpha_{m,n}^*}{\mathbf{E}Z^*_{n}}
&\le 40\frac{2G_n+\lambda_nB_n}{B_{m(n)}}
\le40\frac{(\rho_n+\lambda_n)B_n}{B_{m(n)}}
\le40(\rho_n^{2/3}+\lambda_n^{3/4}).
\end{align*}
Then, combining \eqref{trivial} and \eqref{M5}, we have
\begin{gather*}
\sup_{t\in[B^2_{m(n)},B^2_n]}\left|\frac{U_g(t)}{U_g(B_n^2)}-1\right|
=\max_{m(n)\leq k\leq n}\left|\frac{\mathbf{E}Z^*_{k}}{\mathbf{E}Z^*_{n}}-1\right|
\le\frac{\alpha_{m,n}^*}{\mathbf{E}Z^*_{n}}
\le 40(\rho_n^{2/3}+\lambda_n^{3/4}) \to 0.
\end{gather*}
In particular, $U_g$ is slowly varying since $B_{m(n)}/B_n\to0$.

By a property of slowly varying functions (see, for example,
\cite[p.20]{Sen76})
for every $a>0$ the function
\begin{gather*}
V_a(t):=\frac{\max_{0\le x\leq t}x^aU_g(x)}{t^a}
\end{gather*}
is also slowly varying and $V_a(t)\sim U_g(t)$ as $t\to\infty$. Taking $a=1/3$,
we conclude that
\begin{gather}                                                                                                    \label{oo16}
\max_{1\le k\le n}\frac{B_k^{1/3}\mathbf{E}Z^*_{k}}{B_n^{1/3}\mathbf{E}Z^*_{n}}
\le C_3:=\sup_{t\ge B^2_1}\frac{V_{1/3}(t)}{U_g(t)}<\infty
\quad\text{for all }n\geq1,
\end{gather}
due to the fact that $U_g(t)>0$ for $t\ge B^2_1$.

First, if $n\ge j-1>N_2$ we have from  \eqref{Up2} and  \eqref{oo16} that
$$
\mathbf{P}(T_g>j-1)\le \frac{3 \mathbf{E}Z^*_{j-1}}{B_{j-1}}
\le\frac{3 C_3B_{n}^{1/3}\mathbf{E}Z^*_{n}}{B_{j-1}^{1+1/3}}
\le\left(\frac87\right)^{4/3}\frac{3 C_3B_{n}^{1/3}\mathbf{E}Z^*_{n}}{B_{j}^{4/3}}.
$$
Here we also used \eqref{Up1+}.
Second, for $N_2\ge j-1>0$ we have
$$
\mathbf{P}(T_g>j-1)=\frac{B_j\mathbf{P}(T_g>j-1)}{ \mathbf{E}Z^*_j}
\frac{ \mathbf{E}Z^*_j}{B_j}
\le\frac{B_j}{ \mathbf{E}Z^*_j}
\frac{ C_3B_n^{1/3}\mathbf{E}Z^*_n}{B_j^{4/3}}.
$$
At last, for $j=1$ we have
$$
\mathbf{P}(T_g>0)=\frac{B_1\mathbf{P}(T_g\ge1)}{ \mathbf{E}Z^*_{1}}
\frac{ \mathbf{E}Z^*_{1}}{B_1}
\le\frac{B_1}{ \mathbf{E}Z^*_{1}}
\frac{ C_3B_n^{1/3}\mathbf{E}Z^*_n}{B_1^{4/3}}.
$$
So,   \eqref{oo15} is proved with $C_2:=3(8/7)^{4/3}C_3+C_3\max_{1\le j\le N_2+1}{ B_j}\big/{ \mathbf{E}Z^*_j}<\infty$.
\end{proof}

\begin{lemma}                                                             \label{over5}
For all  $m>N_1$,
\begin{equation}                                                                               \label{oo21}
\beta_m:=\mathbf{E}[Z^*_{\nu_m};Z^*_{\nu_m}> 3B_m]\le 6C_2\mathbf{E}Z^*_mL_m^{2/3}(1).
\end{equation}
\end{lemma}
\begin{proof}
 Note that
\begin{align*}
Z_{\nu_m}&=Z_{\nu_m-1}+g_{\nu_m-1}-g_{\nu_m}+X_{\nu_m}
<B_m+ 2G_m+ X_{\nu_m}<\frac{3}{2}B_m+ X_{\nu_m},
\end{align*}
since $Z_{\nu_m-1}<B_m$ and $2G_m/B_m<1/8<1/2$ by \eqref{Up1}. Hence,
for $1\le j\le m$,
\begin{align*}   \nonumber
\mathbf{E}[Z^*_{\nu_m};\nu_m=j,Z^*_{\nu_m}>3B_m]
&\le \mathbf{E}\left[\frac{3}{2}B_m+X_{j};T_g>\nu_m=j, X_{j}>\frac{3}{2}B_m\right]\\
&\le \mathbf{E}[2 X_{j};T_g>\nu_m=j, X_{j}>B_m]\\
&\le 2\mathbf{E}[ X_{j};T_g>j-1, X_{j}>B_m]\\
&=2\mathbf{E}[ X_{j}; X_{j}>B_m]\mathbf{P}(T_g>j-1)\\
&\le2\mathbf{E}[ X_{j}^2/B_m; X_{j}>B_m]\mathbf{P}(T_g>j-1).
\end{align*}
So,  we have the bound
\begin{gather}                                                                \label{oo22}
\beta_m=\mathbf{E}[Z^*_{\nu_m};Z^*_{\nu_m}> 3B_m]
\leq\frac{2}{B_m}\sum_{j=1}^m \mathbf{E}[X_j^2:|X_j|>B_m]\mathbf{P}(T_g>j-1).
\end{gather}

Now introduce notations
\begin{gather*} 
v_j:=\mathbf{E}[X_j^2:|X_j|>B_n]
\quad \text{and}\quad
V_j:=\sum_{k=1}^j v_k\le B_j^2.
\end{gather*}
We have from  \eqref{oo15} and \eqref{oo22} that
\begin{gather*}  
\beta_m\leq\frac{2}{B_m}\sum_{j=1}^m v_j\mathbf{P}(T_g>j-1)
\le \frac{ 2C_2\mathbf{E}Z^*_m}{B_m^{1-1/3}}
\sum_{j=1}^m \frac{v_j}{B_j^{4/3}}.
\le \frac{2 C_2\mathbf{E}Z^*_m}{B_m^{2/3}}
\sum_{j=1}^m \frac{v_j}{V_j^{2/3}}.
\end{gather*}
It is clear that
$$
\sum_{j=1}^n\frac{v_j}{V^{2/3}_j}
=\sum_{j=1}^m\frac{V_j-V_{j-1}}{V^{2/3}_j}\leq \int_0^{V_m}\frac{dx}{x^{2/3}}=3V^{1/3}_m.
$$
As a result we have
$$
\mathbf{E}[Z^*_{\nu_m};Z^*_{\nu_m}> 2B_m]
\leq 6C_2\mathbf{E}Z^*_m\frac{V^{1/3}_m}{B_m^{2/3}}
=6 C_2\mathbf{E}Z^*_mL_m^{2/3}(1).
$$
This completes the proof of the lemma.
\end{proof}

\begin{proof} [Proof of Theorems \ref{T2} and  \ref{T7}]
First, function $U_g$ is slowly varying by Lemma \ref{over4}.
Second,
by Lemma \ref{over3} we may apply
Proposition \ref{P.upper} with $m=m(n)$. As a result we have from \eqref{Up5} and \eqref{oo21} that
\begin{gather}                                                                                 \label{oo25}
\frac{\beta_{m(n)}}{\mathbf{E}Z^*_n}
\le 6 C_2L_{m(n)}^{2/3}(1)\frac{\mathbf{E}Z^*_{m(n)}}{{\mathbf{E}Z^*_n}}
\le 24 C_2L_{m(n)}^{2/3}(1).
\end{gather}
Note that $B_{m(n)}\ge\lambda_n^{1/4}B_n\ge\lambda_nB_n$ by \eqref{oo1}
and \eqref{Up1}. Thus, using \eqref{lind.cond} and \eqref{Up0}, we obtain
\begin{align*}
\lambda_n^{1/2}B_n^2L_{m(n)}^2(1)
&\le B_{m(n)}^2L_{m(n)}^2(1) =\sum_{k=1}^{m(n)}\mathbf{E}[X_k^2;|X_k|>B_{m(n)}]\\
&\le\sum_{k=1}^n\mathbf{E}[X_k^2;|X_k|>B_{m(n)}]\le\sum_{k=1}^n\mathbf{E}[X_k^2;|X_k|>\lambda_nB_n]\\
&=B_n^2L_n^2(\lambda_n)\le B_n^2\lambda_n^2.
\end{align*}
So, $L_{m(n)}^2(1)\le\lambda_n^{2-1/2}$ and, hence,
 $L_{m(n)}^{2/3}(1)\le\lambda_n^{1/2}$.
Substituting this estimate into \eqref{oo25}  we find from \eqref{oo5} that
\begin{gather*}
\alpha_{n}^*\le72\left(\rho_n^{2/3}+\lambda_n^{1/2}\right)
+24 C_2\lambda_n^{1/2}\qquad\forall \ n>N_3.
\end{gather*}
Thus, the inequality \eqref{oo2} is proved
with
$$
C_1:=72+24C_2+\max\limits_{1\le n\le N_3}\frac{\alpha_{n}^*}{\rho_n^{2/3}+\lambda_n^{1/2}}<\infty.
$$

Next, convergence to $0$ in \eqref{oo2} follows from \eqref{good} and \eqref{lind.cond}  as it was mentioned at the beginning of Subsection 2.3.
This convergence imply equivalence in \eqref{main}.
At last, relations in \eqref{main+} follows from  \eqref{must} and \eqref{Up3} since $G_n/B_n\to0$.
\end{proof}

\subsection{Proof of Theorem~\ref{T1}}
In this subsection we prove weak convergence of the sequence of processes
$
s_n(\cdot) 
$,
conditioned on $\{T_g>n\}$, towards the Brownian meander $M(t),t\in[0,1]$.
Recall that processes $s_n(t)=s(tB_n^2)/B_n,\ t\in[0,1],$
were defined in \eqref{T4.0} and \eqref{T4.0.scaled}.

We shall use the approach from \cite{DW10} which is based on the strong approximation of the broken line
process $s(t)$ by the Brownian motion, see Lemma~\ref{L.pi}.

Let $f:C[0,1]\mapsto\mathbb{R}$ be a non-negative uniformly continuous with respect to the uniform
topology function with values in the interval $[0,1]$. Our purpose is to show that
\begin{equation}
\label{T_Funk.1}
\mathbf{E}[f(s_n)\mid T_g>n]\to\mathbf{E}[f(M)]\quad\text{as }n\to\infty.
\end{equation}

Let $m(n)$ be the sequence defined in \eqref{oo1}. Recall that if $n>N_3$ then $m(n)$ satisfies
all the conditions on pairs $(m,n)$ imposed in Section~\ref{Sect:proof_of_main_result}. Thus, it
follows from \eqref{A12} and \eqref{A21} that
\begin{align}
\label{T_Funk.1a}
\nonumber
Q_{k,n}(y)&\leq \pi_n+4\varphi(0)\varepsilon_{k,n}+\Psi\left(\frac{y}{B_{k,n}}\right)\\
&\leq\rho_n+2\varphi(0)\frac{y}{B_{k,n}}\leq\rho_n+\frac{y}{B_{n}},\quad k\leq m(n).
\end{align}
In particular,
$$
Q_{k,n}(y,0)\leq\frac{2y}{B_n}\quad\text{for all }k\leq m(n)\text{ and }y\geq \rho_nB_n.
$$
Since $B_{m(n)}\geq \rho_nB_n$, we have then by the Markov property,
\begin{align*}
\mathbf{P}(T_g>n,Z^*_{\nu_{m(n)}}>2B_{m(n)})
&=\int_{2B_{m(n)}}^\infty\mathbf{P}(Z^*_{\nu_{m(n)}}\in dy)Q_{\nu_{m(n)},n}(y,0)\\
&\leq\int_{2B_{m(n)}}^\infty\mathbf{P}(Z^*_{\nu_{m(n)}}\in dy)\frac{2y}{B_n}\\
&=\frac{2}{B_n}\mathbf{E}[Z^*_{\nu_{m(n)}};Z^*_{\nu_{m(n)}}>B_{m(n)}].
\end{align*}
Then, in view of Lemma~\ref{over5} and \eqref{main},
\begin{equation*}
\mathbf{P}(T_g>n,Z^*_{\nu_{m(n)}}>2B_{m(n)})=o\left(\mathbf{P}(T_g>n)\right),
\quad n\to\infty
\end{equation*}
and, since $f$ is bounded from above,
\begin{equation}
\label{T_Funk.3}
\mathbf{E}[f(s_n);T_g>n,Z^*_{\nu_{m(n)}}>2B_{m(n)}]=o\left(\mathbf{P}(T_g>n)\right).
\end{equation}

Using \eqref{T_Funk.1a} once again, we have
$$
Q_{m(n),n}(y,0)\leq 2\rho_n^{2/3}\quad\text{for all }y\leq \rho_n^{2/3}B_n.
$$
Therefore,
\begin{align*}
\mathbf{P}(T_g>n,Z^*_{\nu_{m(n)}}\leq \rho_n^{2/3}B_n)
&\leq 2\rho_n^{2/3}\mathbf{P}(0<Z^*_{\nu_{m(n)}}\leq \rho_n^{2/3}B_n)\\
&\leq 2\rho_n^{2/3}\mathbf{P}(Z^*_{\nu_{m(n)}}>0).
\end{align*}
Applying the last inequality in \eqref{Up5} and recalling
that $B_{m(n)}\geq \rho_n^{1/3}B_n$, we get
\begin{align}
\label{T_Funk.4a}
\mathbf{P}(Z^*_{\nu_{m(n)}}>0)\leq 20\frac{\mathbf{E}Z^*_{n}}{B_{m(n)}}
\leq 20\rho_n^{-1/3}\frac{\mathbf{E}Z^*_{n}}{B_{n}}.
\end{align}
Therefore,
\begin{align}
\label{T_Funk.4}
\nonumber
\mathbf{P}(T_g>n,Z^*_{\nu_{m(n)}}\leq \rho_n^{2/3}B_n)
\leq 40\rho_n^{1/3}\frac{\mathbf{E}Z^*_{n}}{B_{n}}=o(\mathbf{P}(T_g>n)).
\end{align}
This implies that
\begin{equation}
\label{T_Funk.5}
\mathbf{E}[f(s_n);T_g>n,Z^*_{m(n)}\leq\rho_n^{2/3}B_{n}]=o\left(\mathbf{P}(T_g>n)\right).
\end{equation}

For every $k\geq0$ and every $y\in\mathbb{R}$ define a functional $f(k,y;\cdot)$ by the following
relation:
$$
f(k,y;h):=f\left(y+\left(h(t)-h\left(\frac{B_k^2}{B_n^2}\right)\right)
\mathbb{I}\left\{t\geq \frac{B_k^2}{B_n^2}\right\}\right),\quad h\in C[0,1].
$$
It follows from the definition of $\nu_{m(n)}$ that
\begin{align*}
\frac{\max_{k\leq\nu_{m(n)}}|S_k-S_{\nu_{m(n)}}|}{B_n}
&\leq \frac{\max_{k\leq\nu_{m(n)}}|Z_k-Z_{\nu_{m(n)}}|}{B_n}+\frac{2G_n}{B_n}\\
&\leq \frac{B_{m(n)}+Z^*_{\nu_{m(n)}}}{B_n}+\frac{2G_n}{B_n}
\leq \frac{2B_{m(n)}}{B_n}+\frac{2G_n}{B_n}
\end{align*}
on the event $\{Z^*_{m(n)}\in(0,2B_{m(n)}]\}$. From this bound and the uniform continuity of the
functional $f$ we infer that
$$
f(s_n)-f\left(\nu_{m(n)},\frac{S_{\nu_{m(n)}}}{B_n},s_n\right)=o(1)
\quad\text{on the event }\{Z^*_{m(n)}\in(0,2B_{m(n)}]\}.
$$
Combining this with \eqref{T_Funk.3} and \eqref{T_Funk.5}, we obtain
\begin{align}
\label{T_Funk.6}
&\mathbf{E}[f(s_n);T_g>n]\\
\nonumber
&=\mathbf{E}\left[f\left(\nu_{m(n)},\frac{S_{\nu_{m(n)}}}{B_n},s_n\right);
T_g>n,Z^*_{\nu_{m(n)}}\in(\rho_n^{2/3}B_n,2B_{m(n)}]\right]
+o\left(\mathbf{P}(T_g>n)\right).
\end{align}
By the Markov property at $\nu_{m(n)}$,
\begin{align*}
&\mathbf{E}\left[f\left(\nu_{m(n)},\frac{S_{\nu_{m(n)}}}{B_n},s_n\right);
T_g>n,Z^*_{\nu_{m(n)}}\in(\rho_n^{2/3}B_n,2B_{m(n)}]\right]\\
&\hspace{1cm}=\sum_{k=1}^{m(n)}\int_{\rho_n^{2/3}B_n}^{2B_{m(n)}}\mathbf{P}(Z^*_{k}\in dy,\nu_{m(n)}=k)\\
&\hspace{3cm}\times\mathbf{E}\left[f\left(k,\frac{y+g_k}{B_n},s_n\right);y+\min_{j\in[k,n]}(Z_j-Z_k)>0\right].
\end{align*}
We now note that it suffices to show that, uniformly in $y\in(\rho_n^{2/3},2B_{m(n)}]$ and $k\leq m(n)$,
\begin{equation}
\label{T_Funk.7}
\mathbf{E}\left[f\left(k,\frac{y+g_k}{B_n},s_n\right);y+\min_{j\in[k,n]}(Z_j-Z_k)>0\right]
=(\mathbf{E}f(M)+o(1))\sqrt{\frac{2}{\pi}}\frac{y}{B_n}.
\end{equation}
Indeed, this relation implies that
\begin{align*}
&\mathbf{E}\left[f\left(\nu_{m(n)},\frac{S_{\nu_{m(n)}}}{B_n},s_n\right);
T_g>n,Z^*_{\nu_{m(n)}}\in(\rho_n^{2/3},2B_{m(n)}]\right]\\
&\hspace{1cm}=\sqrt{\frac{2}{\pi}}\frac{\mathbf{E}f(M)+o(1)}{B_n}
\mathbf{E}[Z^*_{\nu_m(n)};Z^*_{\nu_{m(n)}}\in(\rho_n^{2/3}B_n,2B_{m(n)}]].
\end{align*}
It is clear that
$$
\mathbf{E}[Z^*_{\nu_m(n)};Z^*_{\nu_{m(n)}}\leq\rho_n^{2/3}B_n]
\leq \rho_n^{2/3}B_n\mathbf{P}(Z^*_{\nu_{m(n)}}>0)
$$
Applying \eqref{T_Funk.4a}, we obtain
\begin{align*}
\mathbf{E}[Z^*_{\nu_m(n)};Z^*_{\nu_{m(n)}}\leq\rho_n^{2/3}B_n]
\leq \rho_n^{2/3}B_n\mathbf{P}(Z^*_{\nu_{m(n)}}>0)
\leq 20\rho_n^{1/3}\mathbf{E}Z^*_{n}=o(\mathbf{E}Z^*_{n}).
\end{align*}
Furthermore, by Lemma~\ref{over5} and the second inequality in \eqref{Up5},
$$
\mathbf{E}[Z^*_{\nu_m(n)};Z^*_{\nu_{m(n)}}>2B_{m(n)}]=o(\mathbf{E}Z^*_{n}).
$$
As a result,
$$
\mathbf{E}[Z^*_{\nu_m(n)};Z^*_{\nu_{m(n)}}\in(\rho_n^{2/3}B_n,2B_{m(n)}]]
=(1+o(1))\mathbf{E}Z^*_{n}
$$
and, consequently,
\begin{align*}
&\mathbf{E}\left[f\left(\nu_{m(n)},\frac{S_{\nu_{m(n)}}}{B_n},s_n\right);
T_g>n,Z^*_{\nu_{m(n)}}\in(\rho_n^{2/3},2B_{m(n)}]\right]\\
&\hspace{1cm}=(\mathbf{E}f(M)+o(1))\sqrt{\frac{2}{\pi}}\frac{\mathbf{E}Z^*_{n}}{B_n}
\end{align*}
Plugging this into \eqref{T_Funk.6} and taking into account \eqref{main}, we get
$$
\mathbf{E}[f(s_n);T_g>n]=(\mathbf{E}f(M)+o(1))\mathbf{P}(T_g>n),
$$
which is equivalent to \eqref{T_Funk.1}.

In order to prove \eqref{T_Funk.7}, we apply Lemma~\ref{L.pi}. Set $w_n(t):=W_n(tB_n^2)/B_n$
and define
$$
A_n:=\left\{\max_{t\in[0,1]}|s_n(t)-w_n(t)|\leq\pi_n\right\}.
$$
Then, on this set we have, uniformly in $k$,
$$
\left\|\left(s_n(t)-s_n\left(\frac{B_k^2}{B_n^2}\right)\right)
\mathbb{I}\left\{t\geq \frac{B_k^2}{B_n^2}\right\}-
\left(w_n(t)-w_n\left(\frac{B_k^2}{B_n^2}\right)\right)
\mathbb{I}\left\{t\geq \frac{B_k^2}{B_n^2}\right\}\right\|\leq 2\pi_n.
$$
Since $f$ is uniformly continuous, there exists $\delta_n\to0$ such that
$$
\left|f(k,z;s_n)-f(k,z,w_n)\right|\leq\delta_n\quad\text{on the event }A_n.
$$
Using now \eqref{T_Funk.1a}, we conclude that
\begin{align*}
&\left|\mathbf{E}\left[f\left(k,\frac{y+g_k}{B_n},s_n\right)-f\left(k,\frac{y+g_k}{B_n},w_n\right);
A_n,y+\min_{j\in[k,n]}(Z_j-Z_k)>0\right]\right|\\
&\hspace{2cm}\leq \delta_nQ_{k,n}(y,0)=o\left(\frac{y}{B_n}\right)
\end{align*}
uniformly in $k\leq m(n)$ and $y\in(\rho_n^{2/3}B_n, 2B_{m(n)}]$. From this estimate and
$\mathbf{P}(A_n^c)\leq\pi_n=o(y/B_n)$ we have
\begin{align}
\label{T_Funk.8}
\nonumber
&\mathbf{E}\left[f\left(k,\frac{y+g_k}{B_n},s_n\right);y+\min_{j\in[k,n]}(Z_j-Z_k)>0\right]\\
&\hspace{1cm}
=\mathbf{E}\left[f\left(k,\frac{y+g_k}{B_n},w_n\right);A_n,y+\min_{j\in[k,n]}(Z_j-Z_k)>0\right]
+o\left(\frac{y}{B_n}\right).
\end{align}
On the set $A_n$ we also have
\begin{align*}
&\left\{y-\rho_nB_n+\min_{B_k^2\leq t\leq B_n^2}(W_n(t)-W_n(B_k^2))>0\right\}\\
&\hspace{1cm}\subseteq\left\{y+\min_{j\in[k,n]}(Z_j-Z_k)>0\right\}\\
&\hspace{2cm}\subseteq
\left\{y+\rho_nB_n+\min_{B_k^2\leq t\leq B_n^2}(W_n(t)-W_n(B_k^2))>0\right\}.
\end{align*}
Combining this with \eqref{T_Funk.8} and recalling that $\mathbf{P}(A_n^c)\leq\pi_n=o(y/B_n)$, we obtain
\begin{align}
\label{T_Funk.9}
&\mathbf{E}\left[f\left(k,\frac{y+g_k}{B_n},s_n\right);y+\min_{j\in[k,n]}(Z_j-Z_k)>0\right]\\
\nonumber
&\leq \mathbf{E}\left[f\left(k,\frac{y+g_k}{B_n},w_n\right);
y+\rho_nB_n+\min_{B_k^2\leq t\leq B_n^2}(W_n(t)-W_n(B_k^2))>0\right]+o\left(\frac{y}{B_n}\right)
\end{align}
and
\begin{align}
\label{T_Funk.10}
&\mathbf{E}\left[f\left(k,\frac{y+g_k}{B_n},s_n\right);y+\min_{j\in[k,n]}(Z_j-Z_k)>0\right]\\
\nonumber
&\geq \mathbf{E}\left[f\left(k,\frac{y+g_k}{B_n},w_n\right);
y-\rho_nB_n+\min_{B_k^2\leq t\leq B_n^2}(W_n(t)-W_n(B_k^2))>0\right]+o\left(\frac{y}{B_n}\right).
\end{align}
Since $\rho_nB_n=o(y)$ for $y\geq\rho_n^{2/3}B_n$, we get from \eqref{bm}
$$
\mathbf{P}\left(y\pm\rho_nB_n+\min_{B_k^2\leq t\leq B_n^2}(W_n(t)-W_n(B_k^2))>0\right)
\sim\sqrt{\frac{2}{\pi}}\frac{y}{B_n}.
$$
Furthermore, by Theorem 2.1 in Durrett, Iglehart and Miller \cite{DIM77},
$$
\mathbf{E}\left[f\left(k,\frac{y+g_k}{B_n},w_n\right)\Big|
y\pm\rho_nB_n+\min_{B_k^2\leq t\leq B_n^2}(W_n(t)-W_n(B_k^2))>0\right]\to \mathbf{E}f(M).
$$
Applying these relations to the right hand sides in \eqref{T_Funk.9} and \eqref{T_Funk.10}, we obtain
\eqref{T_Funk.7}. Thus the proof is finished.
\section{Asymptotic properties of $U_g$}

\subsection{Proof of Theorem~\ref{T3}}
If $\overline{g}=\sup_{n\geq1}g_n$ is finite then $\oo{g}-S_{T_g}\ge 0$.
Hence, by the monotone convergence theorem,
\begin{equation*}
E_n:=\mathbf{E}[\oo{g}-S_{T_g};T_g\le{n}]-\oo{g}
\uparrow U_g(\infty)=\mathbf{E}[\oo{g}-S_{T_g}]-\oo{g}\le\infty.
\end{equation*}
Next, from \eqref{M1}  we have
\begin{equation*}%
\label{ff26}
U_g(B_n^2)=\mathbf{E}Z^*_{{n}}=
\mathbf{E}[(\oo{g}-S_{T_g});T_g\le{n}]-\oo{g}+
(\oo{g}-g_{{n}})\mathbf{P}[T_g>n].
\end{equation*}
Using  now \eqref{Up2} and  \eqref{good} we obtain for $n>N_2$ that
\begin{equation*}
\left|\mathbf{E}Z^*_{{n}}-E_n\right|\le (\oo{g}-g_{{n}})\mathbf{P}[T_g>n]
\le o(B_n)\cdot3\mathbf{E}Z^*_n/B_n=o\left(\mathbf{E}Z^*_n\right).
\end{equation*}
Thus
\begin{equation}                                                                                    \label{w0}
0<U_g(B_n^2)=\mathbf{E}Z^*_{{n}}\sim E_n\uparrow U_g(\infty)\le\infty.
\end{equation}

Hence the limit in \eqref{lim+} is well defined. Moreover, the sequence of positive numbers
$\mathbf{E}Z^*_n$ in \eqref{w0} is asymptotically equivalent to the sequence of
non-decreasing numbers $E_n$. Consequently, $E_{N_4}>0$ for some $N_4<\infty$.
Hence, $U_g(\infty)\ge E_{N_4}>0$.

Thus, all assertions of Theorem \ref{T3} are proved because the mentioned there property
of non-increasing sequences $\{g_n\}$ was proved in Remark \ref{rem2}.

Moreover, convergence \eqref{w0} allows us  to obtain
\begin{lemma}                                                             \label{L+1}
If $\bar{g}=\sup\limits_{n\ge 1}g_n<\infty$ then  there exists constant $C_5<\infty$  such that
\begin{gather}                                                                               \label{w1}
B_n\mathbf{P}(T_g>n)\ge C_5>0\quad\text{for all }n\geq1.
\end{gather}
\end{lemma}
\begin{proof}
We have from \eqref{main} and \eqref{w0} that
\begin{equation*}
0<B_n\mathbf{P}(T_g>n)\sim U_g(B_n^2)\sim E_n\uparrow U_g(\infty)\in(0,\infty).
\end{equation*}
This fact implies \eqref{w1}.
\end{proof}


\subsection{Proof of Theorem~\ref{T4}}
We split the proof into two steps.

\begin{lemma}                                                             \label{L+2}
If $\bar{g}<\infty$, then
\begin{gather}                                                                               \label{w2}
2\mathbf{E}[\bar{g}-S_{T_g}]\ge C_5\sum\limits_{k>1}\sigma_k^2(\bar{g}-\bar{g}_k)/B_k^3.
\end{gather}
\end{lemma}
\begin{proof}
We have from  \eqref{w1}  that
\begin{align}
 \nonumber
\mathbf{E}[\bar{g}-S_{T_g}]
&\geq\sum\limits_{k>0}(\bar{g}-g_k)\mathbf{P}(T_g=k)
\ge\sum\limits_{k>0}(\bar{g}-\bar{g}_k)\mathbf{P}(T_g=k)
\\
\nonumber
&=\sum\limits_{k>0}(\bar{g}-\bar{g}_k)[\mathbf{P}(T_g>k-1)-\mathbf{P}(T_g>k)]
\\
\nonumber
&=(\bar{g}-\bar{g}_{1})\mathbf{P}(T_g>0)
+\sum\limits_{k>0}(\bar{g}_k-\bar{g}_{k+1})\mathbf{P}(T_g>k)
\\
\nonumber
&\ge C_5\frac{\bar{g}-\bar{g}_{1}}{B_{1}}
+C_5\sum\limits_{k>0}\frac{\bar{g}_k-\bar{g}_{k+1}}{B_k}
\\
\nonumber
&=C_5\sum\limits_{k>1}(\bar{g}-\bar{g}_k)\left(\frac{1}{B_{k-1}}-\frac{1}{B_k}\right).
\end{align}
But
$$
\frac{1}{B_{k-1}}-\frac{1}{B_k}
=\frac{B_k^2-B_{k-1}^2}{B_k B_{k-1}(B_k+B_{k-1})}
\ge \frac{\sigma_k^2}{2B_k^3}.
$$
So,  \eqref{w2} is proved.
\end{proof}

\begin{lemma}                                                             \label{L+3}
If $\bar{g}<\infty$, then for every $\ee>0$ there exists a constant $N_{4}<\infty$ such that
\begin{gather}                                                                                   \label{w3}
4\mathbf{E}[\bar{g}-S_{T_g}]\ge C_{5}(1-e^{-\varepsilon^2/8})
\sum\limits_{n> N_4}\mathbf{E}[-X_n;-X_n>\varepsilon B_n]/B_{n}.
\end{gather}
\end{lemma}
\begin{proof}
It follows from \eqref{main+} that, for every $\varepsilon>0$,
\begin{gather*}
\mathbf{P}\left(\frac{Z_{n}}{B_n}<\frac{\varepsilon}{2}\Big|T_g>n\right)\to 1-e^{-(\varepsilon/2)^2/2}
=1-e^{-\varepsilon^2/8}>0.
\end{gather*}
Hence, there exists $N_{5}<\infty$ such that
\begin{gather}                                                                                   \label{w4}
\mathbf{P}\left(\frac{Z_{n}}{B_n}<\frac{\varepsilon}{2}\Big|T_g>n\right)
\ge \frac{1-e^{-\varepsilon^2/8}}{2}>0\qquad \text{for all } n\ge N_{5}.
\end{gather}
Using \eqref{good} we find $N_{4}<\infty$ such that $N_4\ge N_5$ and
\begin{gather}                                                                                   \label{w5}
g_{n-1}-g_n<\varepsilon B_n/2\quad \text{for all } n\ge N_{4}.
\end{gather}

Next, since $S_{T_g}=X_{T_g}+Z_{T_g-1}+g_{T_g-1}\le X_{T_g}+Z_{T_g-1}+ \bar{g}$, we have
\begin{gather}                                                                                  \label{w6}
\mathbf{E}[\bar{g}-S_{T_g}]\ge\mathbf{E}[- X_{T_g}-Z_{T_g-1}]=\sum_{n>0}b_n,
\end{gather}
where
\begin{gather}                                                                                  \label{w7}
b_n:=\mathbf{E}[- X_{T_g}-Z_{T_g-1}:T_g=n]
=\mathbf{E}[- X_{n}-Z_{n-1}:T_g>n-1, Z_n\le0].
\end{gather}
Using \eqref{w5} we obtain the following inclusions of events
\begin{gather*}
\{-X_{n}>\ee B_n, Z_{n-1}<\ee B_n/2\}\subset
\{Z_n=X_{n}+Z_{n-1}+g_{n-1}-g_n<0\},
\\
\{-X_{n}>\ee B_n, Z_{n-1}<\ee B_n/2\}\subset\{- X_{n}-Z_{n-1}>-X_n/2\}.
\end{gather*}
Hence, it follows from \eqref{w7} that
\begin{align}                                                                                  \label{w8}
b_n&\ge\mathbf{E}[- X_{n}/2:T_n>n-1,-X_{n}>\ee B_n, Z_{n-1}<\ee B_n/2]
\\ \nonumber
&=\mathbf{E}[- X_{n}/2:-X_{n}>\ee B_n]
\mathbf{P}[T_g>n-1, Z_{n-1}<\ee B_n/2].
\end{align}

Since $B_n>B_{n-1}$, we have from \eqref{w1}, \eqref{w4} and \eqref{w5} that for $n>N_4$
\begin{align*}
\mathbf{P}\left(T_g>n-1, Z_{n-1}<\frac{\ee B_n}{2}\right)
&\ge \mathbf{P}\left(T_g>n-1, Z_{n-1}<\frac{\ee B_{n-1}}{2}\right)\\
&=\mathbf{P}(T_g>n-1)\mathbf{P}\left(Z_{n-1}<\frac{\ee B_{n-1}}{2}\big|T_g>n-1\right)\\
&\ge\frac{C_5(1-e^{-\varepsilon^2/8})}{2B_n}.
\end{align*}
This inequality together with \eqref{w6}, \eqref{w7} and \eqref{w8} imply \eqref{w3}.
\end{proof}

Theorem \ref{T4} immediately follows from Lemmas \ref{L+2} and \ref{L+3}.

\subsection{Proof of Theorem~\ref{T5}} 
Introduce the notation
\begin{align}     \nonumber                                                                   
T&:=T_g,\quad M_n:= h_{n}+\uu_1-\uu_{n},\quad
\oo{M}_m:=\sum_{k>m}M_k\frac{\sigma_k^2}{B^3_k},
\\                                                                                   \label{w11}
H_n&:=h_{n}+g_{n-1}-\uu_{n}>0,\quad
\oo{F}_m:=\sum_{k>m}\frac{1}{B_k}\mathbf{E}[-X_k:-X_k>H_k],
\end{align}
It follows from \eqref{Sum+} and \eqref{hSum} that $\oo{M}_n\to0$,  and $\oo{F}_n\to0$ by \eqref{hLind}. Hence, there exists  finite $N_6$ such that
\begin{align}                                                                                                                                                                   \label{w12}
N_6&:=\min\{m>N_2:\oo{F}_m+\oo{M}_m\le1/8\}<\infty,\quad
\oo{E}_n:=\max_{N_6\le k\le n}\mathbf{E}Z_k^*.
\end{align}

\begin{lemma}\label{L+4}
If $n\ge m\ge N_6$ then
\begin{align}                                                                       \label{w13}
{F}^*_{m}&:=\mathbf{E}[ -X_{T};-X_{T}>H_{T},N_6<T\leq m]
\leq 4\oo{E}_n\oo{F}_{N_6}
\end{align}
and
\begin{align}                                                                         \label{w14}
{M}^*_{m}&:=\mathbf{E}[ M_T;N_6<T\leq m]
\leq 4\oo{E}_n\oo{M}_{N_6}
.
\end{align}
\end{lemma}
\begin{proof}
Since $F_{N_6}^*=M_{N_6}^*=0$ we consider only the case when $n\ge m> N_6$.
First note  that from \eqref{Up2} and \eqref{Up1+} we have
\begin{equation}
\mathbf{P}(T>k)\leq 3\frac{\mathbf{E}Z_k^*}{B_k}\leq 3\frac{\oo E_n}{B_k}\leq 4\frac{\oo E_n}{B_{k+1}}\quad
\mbox{if } n\ge k\ge N_6>N_2.      \label{w15}
\end{equation}
Using \eqref{w15} we obtain
\begin{align*}
F^*_{m}&=\sum^m_{k=N_6+1}\mathbf{E}[\ -X_{k};-X_{k}>H_{k},\ T=k]\\
&\le \sum^m_{k=N_6+1}\mathbf{E}[\ -X_{k};-X_{k}>H_{k},\ T>k-1]\\
&=\sum^m_{k=N_6+1}\mathbf{E}[\ -X_{k};-X_{k}>H_{k}]\mathbf{P}(T>k-1)\\
&\le 4\oo{E}_n\sum_{k>N_6}\frac{1}{B_{k}}\mathbf{E}[\ -X_{k};-X_{k}>H_{k}].
\end{align*}
Now \eqref{w13} follows from \eqref{w11}.

Next, it is easy to see that
\begin{align*}
M^*_{m}&= \sum^{m}_{k=N_6+1}M_k\mathbf{P}(T=k)
 =  \sum^{m}_{k=N_6+1}M_k(\mathbf{P}(T>k-1)-\mathbf{P}(T>k))\\
&=M_{N_6+1}\mathbf{P}(T>N_6)-M_m\mathbf{P}(T>m)
+\sum^{m-1}_{k=N_6+1}(M_{k+1}-M_k)\mathbf{P}(T>k).
\end{align*}
Applying again \eqref{w15} and noting that $\{M_k\}$ is positive and increasing by \eqref{gMonot}, we obtain
\begin{align*}
M^*_{m}&\le
3\oo{E}_n\frac{M_{N_6+1}}{B_{N_6}}
+3\oo{E}_n\sum^{m-1}_{k=N_6+1}\frac{M_{k+1}-M_k}{B_k}\\
&=3\oo{E}_n\sum^{m}_{k=N_6+1}M_k\left(\frac{1}{B_{k-1}}-\frac{1}{B_k}\right)+3\oo{E}_n\frac{M_m}{B_m}\\
&\le3\oo{E}_n\sum_{k>N_6}M_k\left(\frac{1}{B_{k-1}}-\frac{1}{B_k}\right).
\end{align*}
Now, using \eqref{Up1+} we have
\begin{align*}
\frac{1}{B_{k-1}}-\frac{1}{B_k}
&=\frac{B_k^2-B_{k-1}^2}{B_{k-1}B_k(B_{k-1}+B_k)}\\
&\leq \frac{\sigma_k^2}{2B^3_{k-1}}
\leq\left(\frac{8}{7}\right)^{3/2}\frac{\sigma_k^2}{2B_k^3}\leq\frac{4}{3}\frac{\sigma_k^2}{B^3_k}.
\end{align*}
Thus, \eqref{w14} is proved.
\end{proof}

\begin{lemma}\label{L+5}
For all $n> N_6$
\begin{gather}                                                                      \label{w21}
\oo{E}_n\le4C_6-4\uu_1,\qquad
E_n^{(\pm)}:=\mathbf{E}[(\uu_1-S_{T} )^\pm;T\leq n]
\le 3C_6+3|\uu_1|,
\end{gather}
where \ $C_6:=\mathbf{E}[(\uu_1 -S_{T})^+;T\leq N_6]<\infty$.
\end{lemma}
\begin{proof}
To prove this assertion first note that
\begin{align*}
-S_{T}=-S_{T-1}-X_{T}\leq-g_{T-1}-X_{T}
\leq -g_{T-1}+H_{T}-X_{T}\mathbb{I}\{-X_{T}>H_{T}\}
\end{align*}
if only $H_{T}\geq 0$. Hence, with $H_{T}=h_{T}+g_{T-1}-\uu_{T}>0$
we obtain 
\begin{gather*}
(\uu_1-S_{T})^+\le \big(h_{T}+\uu_1-\uu_{T}
-X_{T}\mathbb{I}\{-X_{T}>H_{T}\}\big)^+=M_{T}
-X_{T}\mathbb{I}\{-X_{T}>H_{T}\}
\end{gather*}
with positive right hand side.
Thus, for $m\ge N_6$,
\begin{align*}
E_m^{(+)}&:=\mathbf{E}[(\uu_1 -S_{T})^+;T\leq m]\\
&\le\mathbf{E}[(\uu_1 -S_{T})^+;T\leq N_6]
+\mathbf{E}[M_{T}; N_6<T\le m]\\
&+\mathbf{E}[-X_{T};-X_{T}>H_{T},N_6<T\le m]
=C_6+F_m^*+M_m^*.
\end{align*}

Next, using \eqref{w13}, \eqref{w14} and \eqref{w12} we obtain
\begin{gather*}
E_m^{(+)}\le C_6+4\oo{E}_n(\oo{F}_{N_6}+\oo{H}_{N_6})\le C_6+4\oo{E}_n/8=C_6+\oo{E}_n/2.
\end{gather*}
Now, we have from  \eqref{Up3} that
\begin{align}   \nonumber
0<\frac{3}{4}\mathbf{E}Z_m^*\le\mathbf{E}[ -S_{T};T\leq m]
&=\mathbf{E}[\uu_1 -S_{T};T\leq m]-\uu_1= E_m^{(+)}- E_m^{(-)}-\uu_1\\
\label{w25}
&\le E_m^{(+)}-\uu_1\le C_6-\uu_1+\frac{\oo{E}_n}{2}.
\end{align}
Taking maximum in \eqref{w25} with respect to $m\in[N_6,n]$ we find
\begin{gather*}
\frac{3}{4}\oo{E}_n\le C_6-\uu_1+\frac{\oo{E}_n}{2}.
\end{gather*}
Hence, the first inequality in \eqref{w21}  is proved.

At last, we obtain from  \eqref{w25} with $m=n$ that
\begin{gather*}
E_n^{(+)}\le C_6+\oo{E}_n/2\le 3C_6-2\uu_1,
\quad
E_n^{(-)}< E_n^{(+)}-\uu_1\le 3C_6-3\uu_1.
\end{gather*}
So, all inequalities in  \eqref{w21} are proved.
\end{proof}

Now, from  \eqref{w21} by the monotone convergence theorem we obtain
\begin{gather*}
E_n^{(\pm)}=\mathbf{E}[(\uu_1-S_{T} )^\pm;T\leq n]\uparrow
\mathbf{E}(\uu_1-S_{T} )^\pm\le  3C_6+3|\uu_1|<\infty.
\end{gather*}
Hence, $\mathbf{E}|\uu_1-S_{T}|\le 6C_6+6|\uu_1|<\infty$, and there exists a finite limit
$$
\lim_{n\to\infty} \mathbf{E}[ -S_{T};T\leq n]
=\lim_{n\to\infty}\mathbf{E}[\uu_1 -S_{T};T\leq n]-\uu_1
=\lim_{n\to\infty} E_n^{(+)}-\lim E_n^{(-)}-\uu_1
$$
which is equal to $\lim_{t\to\infty}U_g(t)$ as it follows from  \eqref{main+} .

All assertions of Theorem \ref{T5} are now proved.

\printbibliography

\end{document}